\documentclass[11pt]{article}

\usepackage{amsmath,amssymb}

\usepackage{times}
\usepackage{bm}
\usepackage{natbib}
\usepackage{algorithm}
\usepackage{amssymb}
\usepackage{mathrsfs}
\usepackage{graphicx}
\usepackage{rotating}
\usepackage[flushleft]{threeparttable}
\usepackage{enumitem} 
\usepackage{comment}
\usepackage{mwe}
\usepackage{subfig}
\usepackage{color, colortbl}

\usepackage{star}

\usepackage[colorlinks,
linkcolor=blue,
anchorcolor=blue,
citecolor=blue
]{hyperref}

 \def\hbeta{\hat{\beta}}                \def\hbbeta{\hat{\bfsym \beta}}

 \def\hxi{\hat{\xi}}                    \def\hbxi{\hat{\bfsym {\xi}}}

\newcommand \ti{\tilde}

\newcommand \mbu{\mathbf{u}}

\newcommand \blam{\boldsymbol{\lambda}}
\newcommand \btt{\boldsymbol{\beta}}
\newcommand \hbt{\hat{\btt}}
\newcommand \bttc{\boldsymbol{\beta}^*}

\newcommand \tbt{\tilde{\btt}}

\newcommand{\indep}{\perp \!\!\! \perp}

\def\T{{ \mathrm{\scriptscriptstyle T} }}

\def\##1\#{\begin{align}#1\end{align}}
\def\$#1\${\begin{align*}#1\end{align*}}

\def\T{{ \mathrm{\scriptscriptstyle T} }} 

\newcommand{\Rom}[1]{\text{\uppercase\expandafter{\romannumeral #1\relax}}}
\newcommand*{\rom}[1]{\expandafter\@slowromancap\romannumeral #1@}

\newcommand{\Rmnum}[1]{\expandafter\@slowromancap\romannumeral #1@}

\newcommand{\scolor}[1]{{\color{magenta}#1}}

\newcommand{\gcomment}[1]{\scolor{$\dagger$}\marginpar{\tiny\scolor{WY:\ #1}}\hspace{-3pt}}

\newcommand{\bfsym}[1]{\ensuremath{\boldsymbol{#1}}}

 \def\hbeta{\widehat{\beta}}
 \providecommand{\abs}[1]{\left\lvert#1\right\rvert}
\usepackage{dsfont}
\def\bDelta {\bfsym {\Delta}}
\definecolor{Gray}{gray}{0.9}

\newtheorem{lem}{Lemma}
\newtheorem{dfn}{Definition}
\newtheorem{thm}{Theorem}
\newtheorem{cond}{Condition}
\newtheorem{rmk}{Remark}
\newtheorem{prop}{Proposition}

\usepackage{txfonts}

\usepackage{geometry}
\newcommand{\scad}{\textnormal{SCAD}}
\newcommand{\mcp}{\textnormal{MCP}}

\begin{document}

\title{ A provable two-stage algorithm for penalized hazards regression}     

\author{Jianqing Fan\thanks{Department of Operations Research and Financial Engineering, Princeton University, Princeton, New Jersey, 08544; E-mail: \texttt{jqfan@princeton.edu}.  Research supported by NSF grants DMS-1662139 and DMS-1712591, and NIH grant 5R01-GM072611-15.},~ Wenyan Gong\thanks{Department of Operations Research and Financial Engineering, Princeton University, Princeton, New Jersey, 08544; E-mail: \texttt{wenyang@princeton.edu}.  Research supported by NIH grant 5R01-GM072611-13.}, and Qiang Sun\thanks{Department of Statistical Sciences, University of Toronto, 100 St. George Street, Toronto, ON M5S 3G3, Canada; E-mail: \texttt{qsun@utstat.toronto.edu}.}}


\date{ }

\maketitle

\vspace{-0.25in}

\begin{abstract}
From an optimizer's perspective, achieving the global optimum for a general nonconvex problem is often provably NP-hard using the classical worst-case analysis. In the case of Cox's proportional hazards model, by taking  its statistical model structures  into account, we identify local strong convexity near  the global optimum, motivated  by which we propose to use two convex programs to optimize the folded-concave penalized Cox's proportional hazards regression. Theoretically, we investigate the statistical and computational tradeoffs of the proposed algorithm and establish the strong oracle property of the resulting estimators.  
Numerical studies and a real data analysis lend further support to our algorithm and theory.
\end{abstract}
\noindent
{\bf Keywords}:  Cox's proportional hazards model, counting process, NP-hardness, nonconvexity, sparsity,  survival analysis, variable selection.

\section{Introduction}\label{sec:1}


An important goal of survival analysis is to identify possible risk factors or to evaluate treatment effects in epidemiological studies and clinical trials. To mitigate possible confounding bias, often a large number of covariates, such as clinical variables, imaging phenotypes and genetic markers, are collected and modeled, making the number of covariates far larger than that of observations. 
For inferential tractability and interpretability, a popular approach is to consider the regularized Cox's proportional hazards regression 
\#\label{eq:main}
\hat{\bbeta}=\argmin_{\bbeta\in\mathbb{R}^p}\left\{\cL(\bbeta)+\cP_{\lambda}(\bbeta)\right\},
\#
where $\cL(\bbeta)$ is the negative log partial likelihood function depending on the observed data (to be introduced later) and $\cP_{\lambda}(\bbeta)$ is  a penalty function.


For linear models, there have been a surge of work on penalized regressions in the past two decades. When $\cP_{\lambda}(\cdot)$ is taken to be the $\ell_0$-pseudo norm,
\eqref{eq:main} 
corresponds to the best subset selection with different information criterions. For example, it includes the $C_p$-statistics,  Akaike information criterion,  Bayesian information criterion, minimum description length, and  risk inflation criterion as special cases. Though the $\ell_0$-regularized regression maybe preferred statistically,  it is discrete and thus is  NP-hard to solve \citep{huo2007when}. To alleviate the computational challenge,  some work has been conducted  during the last two decades focusing on convex relaxations. A popular choice is the Lasso  penalty \citep{tibshirani1996regression}. 
In spite of the computational efficiency convex procedures may bring, \cite{fan2001variable} observed that convex surrogates  introduce non-negligible estimation biases. Nonconvex penalties, such as the SCAD \citep{fan2001variable} penalty or MCP \citep{zhang2010nearly},  have been proposed to eliminate the estimation bias for large coefficients and to attain refined statistical rate of convergence under conditions on the minimal signal strength. Theoretical properties have been achieved for the hypothetical global optima
(or some local optima) \citep{fan2001variable, kim2008smoothly, zhang2012general} which is not guaranteed to be achieved by a practical algorithm, such as the coordinate descent algorithm developed by \cite{breheny2011coordinate}. 

A natural question to ask is: is it possible to design a polynomial-time algorithm that can achieve the global optimum? \cite{chen2017strong} give a negative answer by providing a worst-case complexity analysis and showing that problem \eqref{eq:main} with a general convex loss function  and a folded concave penalty is strongly NP-hard. In other words, there does not exist a fully polynomial-time approximation algorithm for solving \eqref{eq:main} with nonconvex penalties. Yet empirical studies have suggested that coordinate descent algorithms for nonconvex penalized regression work well, even better than those for the Lasso  problems. 
So practice seems to contradict theory.

\cite{fan2018lamm} takes a step towards understanding this paradox by adapting the statistical analysis to the algorithmic framework. They propose a sequence of  convex relaxation programs to approximate the original nonconvex optimization problem, and analyzing the statistical properties of the approximate solutions produced by these convex programs.  They show that their  obtained estimator can achieve the oracle rate as if the global optimum of the original nonconvex problem could be obtained. Numerical studies suggest that the proposed algorithm works more stable than the coordinate descent algorithm by \cite{breheny2011coordinate}.

A closer examination at the geometry around the true regression coefficient $\bbeta^*$ demonstrates that the original nonconvex regression problem becomes  (strictly) convex in a locally restricted neighborhood of $\bbeta^*$:
\begin{align}\label{eq:ncvx:2}
\big\{\bbeta~\textnormal{is sparse} ~\textnormal{and}~\|\bbeta-\bbeta^*\|_1\leq r\big\},
\end{align}
for $r$ sufficiently small, provided the loss function satisfies the local sparse strong convexity condition such that the convexity of the loss function can dominate the concavity of penalty function. Therefore, if we run an iterative algorithm starting from an initial estimator in the region of \eqref{eq:ncvx:2}, the algorithm will finally converge to the (unique) global optimum. This observation suggests a two-stage optimization algorithmic framework for any nonconvex problem sharing a similar landscape:
\begin{enumerate}
\item In the first stage, run a convex relaxation to find a good initial estimator in the locally restricted neighborhood of the underlying true regression parameter (we refer to this stage  as the burn-in stage);
\item In the second stage, run an iterative algorithm that can keep the solution sequence in the locally restricted neighborhood until convergence (to the global optimum).
\end{enumerate} 
\cite{fan2018lamm}'s algorithm happens to fall in this two-stage algorithmic framework as the first local linear approximation problem is used to burn in while all the remaining  $\ell\geq 2$ local linear approximation problems are used for global convergence. This encourages us to investigate the properties of this two-stage algorithmic framework for estimating the Cox's proportional hazards regression model.

Inspired by the above intuition, we propose to directly optimize the second stage using the local adaptive majorization principle instead of adopting a sequence of convex programs. This helps to further reduce the  iteration complexity of the algorithm developed in \cite{fan2018lamm}. Theoretically, we prove the oracle properties of obtained estimators. Numerical studies suggest that our proposed algorithms work more stable than directly running the coordinate descent, thanks to the burn-in stage of the algorithm.

The rest of paper proceeds as follows. In Section \ref{sec:method}, we introduce the Cox's proportional hazards model and a two-stage algorithm called TLAMM. In Section \ref{sec:eigen}, we discuss some conditions on the localized eigenvalues of the Hessian matrix and prove the localized property of Cox's Hessian.
Section \ref{sec:twostage} establishes the statistical and algorithmic property  of the estimator resulting from TLAMM. Numerical experiments are used to examine the finite-sample performance of  the proposed algorithm in Section \ref{sec:simu}.  In Section \ref{sec:data}, we apply TLAMM to The Cancer Genome Atlas (TCGA) skin cutaneous melanoma dataset to study the genes that are associated with the survival of melanoma patients. Section \ref{sec:dis} concludes the paper with a brief discussion. \\

\noindent{\bf Related Work}: \cite{tibshirani1997Lasso} proposed to use the Lasso  penalty in Cox's model for simultaneous parameter estimation and variable selection. \cite{gui2005penalized} used the least angle regression  \citep{efron2004least}  algorithm to compute the Lasso  estimator in Cox's model  and applied to a microarray gene expression dataset. To reduce the bias of  the Lasso  estimator, \cite{fan2002variable} applied the SCAD regularization to Cox's model which allows the addition of an frailty term and showed that the resulting estimator achieved the oracle performance when the dimensionality is fixed. Later, \cite{zhang2007adaptive} proposed to use adaptive Lasso  to estimate the Cox's model, which achieved a similar oracle rate without nonconvexity issues. \cite{zou2008note}  proposed a path-based variable selection method which is consistent for variable selection and efficient in estimation with a proper choice of shrinkage parameter. \cite{wang2009hierarchically} developed an effective and flexible method for group selection in Cox's model. In \cite{antoniadis2010dantzig}, the authors 
studied the statistical properties of the Dantzig selector in Cox's model. \cite{du2010penalized} studied the Cox's model with semiparametric relative risk and proposed a procedure where two penalties are sequentially applied to achieve the oracle property. In \cite{bradic2011regularization}, they focus on Cox's model with ultrahigh dimensionality and establish strong oracle property for  nonconcave penalized methods. \cite{huang2013oracle} and \cite{kong2014non} both studied the oracle inequalities of Lasso  estimators in Cox's model under different set of conditions. \\

\noindent{\bf Notation:} We summarize the notations that will be used throughout the paper. Bold font is used for all vectors and matrices. For any vector $\bu = (u_1, \ldots, u_d)^\T \in \RR^p$ and $q \geq 1$, $\norm{\bu}_q=\big(\sum_{j=1}^p |u_j|^q\big)^{1/q}$ is the $\ell_q$ norm. For any vectors $\bu , \bv\in \RR^p$, we write $\langle \bu, \bv \rangle = \bu^\T \bv$. Moreover, we let $\norm{\bu}_0 = \sum_{j=1}^p \mathds{1}_{\{u_j \neq 0 \}}$ to denote the number of nonzero entries of $\bu$, and set $\|\bu\|_\infty=\max_{1\leq j\leq p}|u_j|$. For two sequences of real numbers $\{ a_n \}_{n\geq 1}$ and $\{ b_n \}_{n\geq 1}$, $a_n \lesssim b_n$ denotes $a_n \leq C b_n$ for some constant $C>0$ independent of $n$, $a_n \gtrsim b_n$ if $b_n \lesssim a_n$, and $a_n \asymp b_n$ signifies that $a_n \lesssim b_n$ and $b_n \lesssim a_n$. If $\bA$ is an $m\times n$ matrix, we use $\| \bA \|_q$ to denote its order-$q$ operator norm, defined by $\| \bA \|_q = \max_{ \bu \in \RR^n} \norm{\bA \bu}_q/\norm{\bu}_q$. For a set $S$, we use $\abs{S}$ to denote its cardinality. 


\section{Methodology}\label{sec:method}

Let $T,\, C\in\mathbb{R}$ be the observed follow-up time and the censoring time respectively.  Denote by $Z=\min\{T,C\}$ the observed failure time and by $\delta=I(T\leq C)$ the censoring indicator. We assume the following censoring mechanism: $T$ and $C$ are independent given the covariates $\bx\in \RR^p$. Suppose we have collected $\left\{(\bX_i,Z_i,\delta_i): i=1,2,...,n\right\}$ such that they are i.i.d copies of $(\bX,Z,\delta)$.

Let $\lambda(t|\bx)$ be the conditional hazard rate function at time $t$ given the covariates $\bX=\bx$. The Cox's proportional hazards model assumes that 
\begin{equation}\label{cox:notime}
\lambda(t|\bx):=\lim\limits_{\Delta_t\rightarrow0} P\left(t<T<t+\Delta_t\vert T>t,\bx\right)/\Delta_t=\lambda_0(t)\exp(\bx^\top\bbeta),
\end{equation}
where  $\bbeta$ is the vector of log hazard ratios (HR) and $\lambda_0(\cdot)$ is the baseline hazard function. Since some covariates can be time-dependent such as age, weight and blood pressure, we consider  the time-varying version of \eqref{cox:notime} by assuming
\begin{equation}
\lambda(t|\bX(t))=\lambda_0(t)\exp(\bbeta^{\top}\bX(t)).
\end{equation}
In this paper, we consider  the general case where the covariates can be   left-continuous.

Following \cite{fleming1991counting}, the negative log-partial likelihood is
\begin{equation}\label{eqloss1}
\cL(\bbeta)=\frac{1}{n}\sum\limits_{i=1}^n\int_{0}^{\infty}\left\{\bbeta^{\top}\bX_i(t)-\log\left(\sum\limits_{i=1}^n Y_i(t)\exp\left(\bbeta^{\top}\bX_i(t)\right)\right)\right\}\mathrm{d}N_i(t),
\end{equation}
where $N_i(t)=\mathds{1}\{Z_i\leq t, \delta_i=1\}$ is the counting process corresponding to the observed failures 
 and $Y_i(t)=\mathds{1}\{Z_i\geq t\}$ is the at-risk process.



In high dimensions, we  consider the following  penalized partial-likelihood estimation problem 
\begin{equation}\label{eqloss}
\widehat\bbeta=\argmin_{\bbeta}\left\{ \cL(\bbeta)+\sum\limits_{k=1}^p p_{\lambda}\left(\abs{\beta_k}\right)\right\},
\end{equation}
where $p_{\lambda}(\cdot)$ is a  folded-concave penalty function such as SCAD and MCP, and  $\lambda$ is a non-zero regularization parameter.
We assume that the underlying regression coefficient vector $\bbeta^*$ is sparse  with support set $\cO$ such that $\abs{\cO}=s$.



%



From the computational perspective, minimizing the nonvex penalized loss function \eqref{eqloss} is challenging  due to its intrinsic nonconvex structure.  \cite{chen2017strong} exploited the worst-case analysis to show that solving \eqref{eqloss} with a general convex loss function and a nonconvex penalty such as the SCAD and MCP  is strongly NP-hard, indicating that, in general, there does not exist a polynomial-time algorithm for solving \eqref{eqloss}.
However, empirical  studies have suggested that various algorithms, such as the local linear approximation  \citep{zou2008one} and the coordinate descent  \citep{breheny2011coordinate}, perform favorably despite the nonconvexity issues.

We examine this paradox closely in this section by looking at the landscape of the loss function  around the true regression coefficient vector $\bbeta^*$.  Obviously, it is impossible to estimate $\bbeta^*$ without further conditions: the Hessian matrix $\nabla^2\cL(\bbeta)$ is singular in high dimensions, rendering non-identifiability issues.  A common remedy  is to assume some local invertibility condition of the Hessian matrix \citep{candes2007dantzig,bickel2009simultaneous}. Roughly speaking, we assume that, for some working sparsity $\bar s$ and some radius $r$, the Hessian matrix is sparsely invertible in the following local zone 
\$
\cC(\bar s, r):= \left\{  \bbeta:  \|\bbeta\|_0\leq \bar s, \, \|\bbeta-\bbeta^*\|_1\leq r \right\},
\$
or more precisely, there exists some $\rho_* > 0$ such that
\$
\min_{ \|\ub\|_0\leq \bar s,\bbeta \in \cC(\bar s,r)} \left\{\frac{\ub^\T \nabla^2 \cL(\bbeta)\ub}{\ub^\T\ub}\right\}\geq \rho_*.
\$

When the nonconvex penalty  is properly tuned such that $\rho_*$ is larger than the maximum concavity of $p_\lambda(\cdot)$, problem \eqref{eqloss} becomes convex in the local restricted region of $\cC(\bar s, r)$. For example, if $p_\lambda(\cdot)$ is the SCAD penalty, we have 
\$
\max_{\beta}\{ -p''_{\lambda}(\beta)\}=\frac{1}{a-1}.
\$
By taking $a>1+1/\rho_*$, problem \eqref{eqloss} becomes sparse strongly convex for any $\bbeta\in\cC(\bar{s}, r)$. We define a shifted loss function as in \cite{loh2017support} such that $\ti{\cL}(\bbeta)=\cL(\bbeta)+p_{\lambda}(\bbeta)-\lambda\norm{\bbeta}_1$. To proceed, we rewrite the objective function as
\begin{equation}
\cL(\bbeta)+p_{\lambda}(\bbeta)=\ti{\cL}(\bbeta)+\lambda\norm{\bbeta}_1,
\label{eq:shiftedloss}
\end{equation}
 which is  convex in $\cC(\bar s,r)$.  Therefore, if we can propose an algorithm that starts  from an initialization in $\cC(\bar s, r)$ and keeps the solution sequence in $\cC(\bar s, r)$, then optimizing  \eqref{eqloss} is equivalent to optimizing a convex problem. A natural question is: How shall we find a good initialization $\widehat\bbeta^1$ such that $\widehat\bbeta^1\in \cC(\bar s, r)$? This can be done by solving a convex relaxation of problem \eqref{eqloss} - the Lasso problem, which is the first stage of the algorithm:
\#\label{eq:Lasso }
\widehat\bbeta^1=\argmin_{\bbeta}\left\{ \cL(\bbeta)+\sum\limits_{k=1}^p \lambda\abs{\beta_k}\right\}.
\#
Note that this optimization is the same as the local linear approximation  \citep{zou2008one} to the problem \eqref{eqloss} starting at the initial value $0$.
The first stage is a convex problem and thus can be solved efficiently. Then starting from $\widehat\bbeta^1$ we can optimize \eqref{eqloss} directly. In both stages, we apply the Local Adaptive Majorize-Minimization (LAMM) algorithm \citep{fan2018lamm} to solve the corresponding optimization problem. The algorithm is thus referred to as the Two-stage LAMM (TLAMM) algorithm.   Figure \ref{fig:1} shows an illustration of the TLAMM algorithm.

\begin{figure}[t!]
\centering
\includegraphics[scale=0.35]{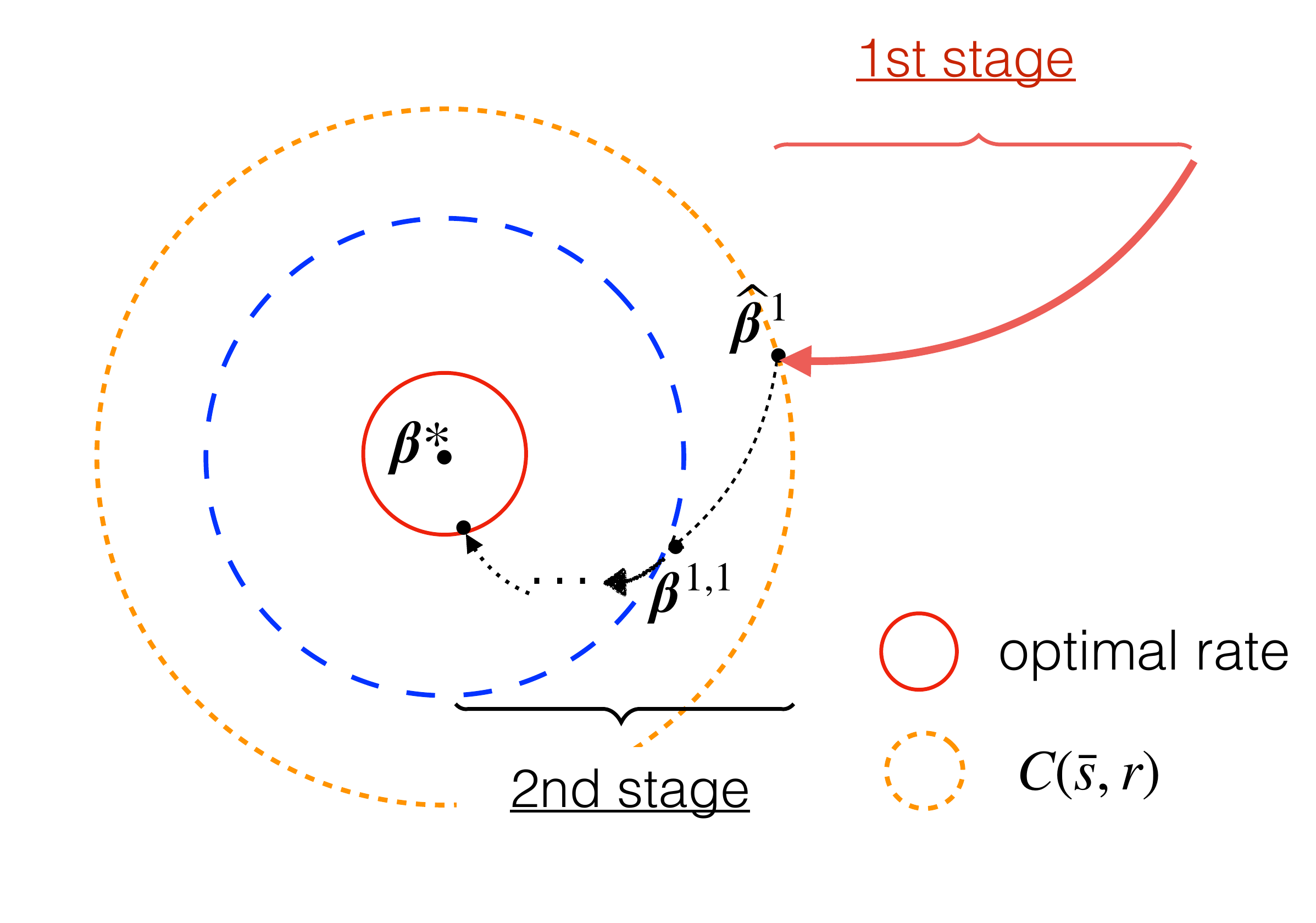}
\caption{An illustration of TLAMM.}
\label{fig:1}
\end{figure}

To fix idea, we describe the second stage in details. The idea also applies to the first stage.  At any working solution $\bbeta^{2,k}$, we locally majorize $\ti{\cL}(\bbeta)$ at $\bbeta^{2,k}$ by the isotropic quadratic function
\begin{equation}
\Psi_{\ti{\cL},\phi^{2,k+1}}(\bbeta,{\bbeta}^{2,k}):=\ti{\cL}({\bbeta}^{2,k})+\left\langle\nabla\ti{\cL}({\bbeta}^{2,k}),\bbeta-{\bbeta}^{2,k}\right\rangle+\frac{\phi^{2,k+1}}{2}\norm{\bbeta-{\bbeta}^{2,k}}_2^2,
\label{eq:psi}
\end{equation}
with $\phi^{2,k+1}$ chosen (to be discussed below) such that the next-step update $\bbeta^{2,k+1}$ satisfies
\#\label{eq:ls}
\Psi_{\ti{\cL},\phi^{2,k+1}}(\bbeta^{2,k+1},{\bbeta}^{2,k})\geq \ti{\cL}(\bbeta^{2,k+1}).
\#
With the majorization \eqref{eq:psi}, the next-step update $\bbeta^{2,k+1}$ is given by
\begin{equation}\label{eq:minimizer}
\bbeta^{2,k+1}=\argmin_{\bbeta\in \RR^d} \left\{\ti{\cL}({\bbeta}^{2,k})+  \left\langle\nabla\ti{\cL}({\bbeta}^{2,k}),\bbeta-{\bbeta}^{2,k}\right\rangle+\frac{\phi^{2,k+1}}{2}\norm{\bbeta-{\bbeta}^{2,k}}_2^2+ \lambda\norm{\bbeta}_1  \right\}.
\end{equation}
The solution to \eqref{eq:minimizer} has a closed-form updating rule:
\#\label{eq:st}
\bbeta^{2,k+1}=S\left({\bbeta}^{2,k}-\frac{1}{\phi^{2,k+1}}\nabla\ti{\cL}\big({\bbeta}^{2,k}\big),\lambda/\phi^{2,k+1}\right)=:T_{\ti{\cL},\lambda,\phi^{2,k+1}}(\bbeta^{2,k}),
\#
and $S(x,\lambda)=\text{sign}(x)\cdot\max\left(\abs{x}-\lambda,0\right)$ is the soft-thresholding operator.

The quadratic coefficient $\phi^{2,k+1}$ can be chosen by a line-search method \citep{beck2009fast}. We can start from a small factor $\phi_0=10^{-4}$, solve for  \eqref{eq:st} and plug it into \eqref{eq:ls} to check whether the local majorization condition hold. If yes, the algorithm outputs $\phi^{2,k+1}=\phi_0$ and the corresponding $\bbeta^{2,k+1}$ ; otherwise, we inflate its value by multiplying a fixed scale $\gamma_u>1$ and repeat the above steps until the local majorization condition hold.  Such a solution always exists since (2.7) with a sufficiently large $\phi$ will eventually majorize $\ti{\cL}(\bbeta)$. 
Following \cite{fan2018lamm}, we refer to this line search method as the LAMM algorithm, which is summarized in the box of Algorithm \ref{alg:ls}. 


By repeating the LAMM algorithm,  a sequence of solutions $\{{\bbeta}^{2,k}:k=0,1,2,...\}$ are generated.  To stop the algorithm, we make use of the first order optimality condition. According to \eqref{eqloss} and \eqref{eq:shiftedloss}, if $\hat{\bbeta}^{2}$ is the second-stage minimizer, it must satisfy
\begin{equation}
\nabla\ti{\cL}(\widehat{\bbeta}^{2})+\lambda\bxi={\bf 0}\quad\text{for}\quad\bxi\in\partial \norm{\hbbeta^2}_1.
\end{equation}
Inspired by this, we stop the optimization algorithm  when
\begin{equation}
\omega_{\lambda}(\bbeta^{2,k}):=\min\limits_{\bxi\in\partial \norm{\bbeta^{2,k}}_1}\left\{\norm{\nabla\ti{\cL}(\bbeta^{2,k})+\lambda\bxi}_{\infty}\right\}\leq\varepsilon,
\label{eq:stopping}
\end{equation}
where $\varepsilon$ is a prefixed optimization error. We call $\ti{\bbeta}^2:=\bbeta^{2,k}$ an $\varepsilon$-optimal solution. A detailed description of TLAMM could be found in the box of Algorithm \ref{alg:tlamm}.   
\begin{remark}
We prove  in Lemma \ref{lem:stopping2} that there exists constants $C>0$ such that
\$
	\omega_{\lambda} (\btt^{\ell,k})\leq C\big\|\btt^{\ell,k}-\btt^{\ell, k-1}\big\|_2.
	\$
for $\ell=1$ and $2$. Thus, in practice,  we stop the algorithm when consecutive solutions are close enough.
\end{remark}


\begin{algorithm}[!t]\label{alg:1}
    \caption{The LAMM algorithm  in the $k$-th iteration of the $\ell$-th stage ($\ell = 1, 2$). }\label{alg:ls}
    \begin{algorithmic}[1]
        \STATE{{\bf Algorithm}: $\{\btt^{\ell,k+1},\phi^{\ell,k+1} \} \leftarrow \mbox{LAMM}(\ell,\lambda,\btt^{\ell,k}, \phi_{0},\phi^{\ell,k}) $  }
        \STATE{\textbf{Input}: $\ell,\lambda, \btt^{\ell,k}, \phi_{0}, \phi^{\ell,k}$   }
        \STATE{\textbf{Initialize}: $\phi^{\ell,k+1}\leftarrow \max\{\phi_{0},\gamma_u^{-1}\phi^{\ell,k}\}$  }
        \STATE{\textbf{Repeat}}
        \STATE{$\quad{}$ \textbf{If} $\ell=1$\textbf{ then } $\btt^{\ell,k+1}\leftarrow T_{\cL,\lambda,\phi^{\ell,k+1}}(\btt^{\ell,k})$ }
        \STATE{$\quad\quad{}$ \textbf{If} $\cL(\bbeta^{\ell,k+1})> \Psi_{\cL,\phi^{\ell,k+1}}(\btt^{\ell,k+1};\btt^{\ell,k})$ \textbf{ then } $\phi^{\ell,k+1}\leftarrow \gamma_u\phi^{\ell,k+1}$ }
        \STATE{$\quad{}$ \textbf{If} $\ell=2$ \textbf{ then } $\btt^{\ell,k+1}\leftarrow T_{\ti{\cL},\lambda,\phi^{\ell,k+1}}(\btt^{\ell,k})$ }
        \STATE{$\quad\quad{}$ \textbf{If} $\ti{\cL}(\bbeta^{\ell,k+1})> \Psi_{\ti{\cL},\phi^{\ell,k+1}}(\btt^{\ell,k+1};\btt^{\ell,k})$ \textbf{ then } $\phi^{\ell,k+1}\leftarrow \gamma_u\phi^{\ell,k+1}$ }
        \STATE{ \textbf{Until}  $\cL(\btt^{\ell,k+1})\leq \Psi_{\cL,\lambda,\phi^{\ell,k+1}}(\btt^{\ell,k+1};\btt^{\ell,k})$ \textbf{If} $\ell=1$}
        \STATE{ $\quad\quad\text{ }\text{ }$$\ti{\cL}(\btt^{\ell,k+1})\leq \Psi_{\ti{\cL},\lambda,\phi^{\ell,k+1}}(\btt^{\ell,k+1};\btt^{\ell,k})$ \textbf{If} $\ell=2$}
        \STATE{\textbf{Return} $\{\btt^{\ell,k+1},\phi^{\ell,k+1}\}$}
    \end{algorithmic}
\end{algorithm}

\begin{algorithm}[ht]
	\caption{ The TLAMM algorithm.}\label{alg:tlamm}
	\begin{algorithmic}[1]
		\STATE{\textbf{Algorithm}: $\ti{\btt}^{2} \leftarrow \text{I-LAMM}(\lambda, \btt^{1,0})$ }
		\STATE{\textbf{Input}: $\phi_{0}>0$}
		\FOR {$k=0, 1, 2, \cdots ~\text{until}~ \norm{\bbeta^{1,k}-\bbeta^{1,k+1}}_2$ is sufficiently small}
		\STATE{$\{\btt^{1,k+1},\phi^{1,k+1} \} \leftarrow \mbox{LAMM}(1,\lambda,\btt^{1,k}, \phi_{0}, \phi^{1,k}) $ }
		\ENDFOR
		\STATE{\textbf{Set}: $\bbeta^{2,0}=\ti{\bbeta}^{1}=\bbeta^{1,k+1}$}
		\FOR {$k=0, 1, 2, \cdots ~\text{until}~ \norm{\bbeta^{2,k}-\bbeta^{2,k+1}}_2$ is sufficiently small}
		\STATE{$\{\btt^{2,k+1},\phi^{2,k+1} \} \leftarrow \mbox{LAMM}(2,\lambda,\btt^{2,k}, \phi_{0}, \phi^{2,k}) $ }
		\ENDFOR		
	     \STATE{\textbf{Output}: $\widetilde\bbeta^{2}=\bbeta^{2, k+1}$ }
	\end{algorithmic}
\end{algorithm}

\begin{section}{Localized Sparse Eigenvalues}\label{sec:eigen}

In this section, we study the local geometry of the Cox's loss function by introducing the localized sparse eigenvalue (LSE) and the corresponding condition. We also verify that the loss function of Cox's model satisfies an LSE condition which suggests the localized strong convexity around $\bbeta^*$. 
Recall that $\norm{\bbeta^*}_0=s$.

\begin{dfn}(Localized Sparse Eigenvalue (LSE))
The maximum and minimum localized sparse eigenvalues of $\cL(\cdot)$ are defined as
\begin{gather*}
\rho_+(m,r)=\sup\limits_{\bu,\bbeta}\left\{\bu_{\texttt{J}}^{\top}\nabla^2\cL(\bbeta)\bu_{\texttt{J}}:\norm{\bu}_2^2=1,\norm{\bu}_0\leq m,\norm{\bbeta-\bbeta^*}_1\leq r\right\};\\
\rho_-(m,r)=\inf\limits_{\bu,\bbeta}\left\{\bu_{\texttt{J}}^{\top}\nabla^2\cL(\bbeta)\bu_{\texttt{J}}:\norm{\bu}_2^2=1,\norm{\bu}_0\leq m,\norm{\bbeta-\bbeta^*}_1\leq r\right\}.
\end{gather*}
\end{dfn}

We define the maximum and minimum localized sparse eigenvalues of $\ti{\cL}(\cdot)$ in the same way and denote them as $\kappa_+(m,r)$ and $\kappa_-(m,r)$. Due to the concavity of the term $p_{\lambda}(\bbeta)-\lambda\norm{\bbeta}_1$, we have $\kappa_-(m,r)\leq \rho_-(m,r)\leq \kappa_+(m,r)\leq\rho_+(m,r)$.

\begin{cond}\label{cond:lse}
We say the LSE condition holds if, for some given constant $C_1$ and radius $r$,  there exists an integer $\tilde{s}\gtrsim s$   such that
\begin{gather*}
0<\rho_*\leq\kappa_-(2s+2\tilde{s},r)\leq\rho_+(2s+2\tilde{s},r)\leq\rho^*<+\infty,\\
\tilde{s}/s> C_1\rho^2_+({2s}+2\tilde{s},r)/\kappa^2_-(2s+2\tilde{s},r).
\end{gather*}
\end{cond}


Empirically $\ti{s}$ is of the same order as $s$. Our next theorem suggests that Condition \ref{cond:lse} holds with high probability when the regularization is properly parametrized.
\begin{thm} Suppose $\{\bX_i(t),Y_i(t),t\geq 0\}$ are i.i.d. processes from $\{\bX(t),Y(t),t\geq 0\}$ with  $\mathbb{P}\{\sup\limits_t\norm{\bX_i(t)}_{\infty}\leq M\}=1$ for a constant $M>0$. Assume that the maximum event time $t^*<+\infty$ and let $r_*=\mathbb{E}\left[Y(t^*)\exp(\bbeta^{*\top}\bX(t^*))\right]$. Then, for any $s'\leq p$, we have
\begin{gather*}
\rho_+(s',r)\leq \exp(4rM)\left\{C_+(s')+4s'M^2\left[\left(1+\Lambda_0(t^*)\right)L_n\left(\frac{p(p+1)}{\varepsilon}\right)+\frac{2}{r_*}\Lambda_0(t^*)t_{n,p,\varepsilon}^2\right]\right\},\\
\rho_-(s',r)\geq \exp(-4rM)\left\{C_-(s')-4s'M^2\left[\left(1+\Lambda_0(t^*)\right)L_n\left(\frac{p(p+1)}{\varepsilon}\right)+\frac{2}{r_*}\Lambda_0(t^*)t_{n,p,\varepsilon}^2\right]\right\}
\end{gather*}
hold with probability at least $1-\exp(-nr_*^2/(8M^2))-2\varepsilon$. 
 Here, $\Lambda_0(t)=\int_0^t \lambda_0(u)\mathrm{d}u$ is the cumulative baseline hazard function, $C^*\geq C_+(s')\geq C_-(s')\geq C_*>0$ where $C^*$ and $C_*$ are two constants depending on $\bbeta^*$, $L_n(t)=\sqrt{(2/n)\log(t)}$ and $t_{n,p,\varepsilon}$ is the solution to $p(p+1)\exp(-nt_{n,p,\varepsilon}^2/(2+2t_{n,p,\varepsilon}/3))=\varepsilon/3$.
\label{thm:minse}
\end{thm}


For any $s'\leq p$, if $n\asymp s'^2\log p$, then $L_n\left(p(p+1)/\varepsilon\right)\asymp1/s'$ and $t_{n,p,\varepsilon}\lesssim1/s'$. Take $s'=2s+2\ti{s}$ and treat the cumulative baseline hazard $\Lambda_0(t^*)$ as a constant. Theorem \ref{thm:minse} implies that if $n\geq C(2s+2\tilde{s})^2\log p$ for a sufficiently large constant $C$, then with high probability, $\rho_-(2s+2\tilde{s},r)$ is lower bounded by $C_*\exp(-4rM)/2$ and $\rho_+(2s+2\tilde{s},r)$ is upper bounded by $2C^*\exp(4rM)$. As long as the second order derivative of the regularization $p'_{\lambda}(\cdot)$ is larger than $-C_*\exp(-4rM)/2$, $\kappa_-(2s+2\ti{s})$ is also bounded below by a positive value. Meanwhile, for large enough $\tilde{s}$, $\rho_+(2s+2\tilde{s},r)/\kappa_-(2s+2\tilde{s},r)$ remains constant as $\tilde{s}$ grows.  Thus there always exists an $\tilde{s}$ such that $\tilde{s}/s\geq C_1\rho^2_+({2s}+2\tilde{s},r)/\kappa^2_-(2s+2\tilde{s},r)$. Therefore, Condition \ref{cond:lse} holds with high probability.



\end{section}

\section{Theoretical Results}\label{sec:twostage}

To present the main theorem, we first need a condition on the folded concave penalty function.

\begin{cond}\label{cond:concave}
The penalty function satisfies 
\begin{enumerate}
\item $p'_{\lambda}(\cdot)$ is a non-increasing continuous function defined on $[0,\infty)$;
\item $\lim_{x\rightarrow 0+}p'_\lambda(x)=\lambda$;
\item There exists a constant $a_1>0$ such that $p'_{\lambda}(\beta)=0$ when $\beta>a_1\lambda$.
\end{enumerate}
\end{cond}

Condition \ref{cond:concave} holds for two mainstream folded concave penalties: SCAD and MCP. For SCAD penalty, $a_1=a$; for MCP, $a_1=\gamma$. Here, both $a$ and $\gamma$ are user-picked parameters associated with the penalty functions, e.g. $a=3.7$ for SCAD and $\gamma=3.0$ for MCP are suggested in the papers where they were first proposed. 


\subsection{Statistical Properties}

In TLAMM, we iteratively use LAMM to solve an optimization problem with localized linear approximation in (\ref{eq:Lasso }) in the first stage and then directly optimize the penalized loss function (\ref{eqloss}) in the second stage. In this section, we prove statistical theory for the first-stage $\varepsilon_1$-optimal estimator $\ti{\bbeta}^{1}$ and the second-stage $\varepsilon_2$-optimal estimator $\ti{\bbeta}^{2}$ in Cox's model. Indeed, when  $\varepsilon_1$ and $\varepsilon_2$ are chosen properly, we shall prove that $\ti{\bbeta}^{1}$ is within $\cC(s+\ti{s},r)$ with high probability, based on which the second-stage estimator $\ti{\bbeta}^2$ can achieve the oracle property.

\begin{prop}\label{prop:contraction}
Suppose that Condition \ref{cond:lse} holds. If $\lambda,\varepsilon_1$ and $r$ satisfy
\begin{equation} \label{fan1}
2(\norm{\nabla(\cL(\bbeta^*))}_{\infty})+\varepsilon_1\leq\lambda\leq s^{-1}r\rho_*/36,
\end{equation}
any $\varepsilon_1$-optimal solution $\tilde{\bbeta}^1$ satisfies
\begin{equation}
\norm{\tilde{\bbeta}^1-\bbeta^*}_2\leq 18\rho_*^{-1}\lambda\sqrt{s}\quad\text{and}\quad\norm{\ti{\bbeta}^1-\bbeta^*}_1\leq 36\rho_*^{-1}\lambda s.
\label{eq:stage1}
\end{equation}
\end{prop}

This is a deterministic statement that bounds the estimation error after the first stage. With a properly selected $\lambda$, $\ti{\bbeta}^1$ is within an $\ell_1$ ball with radius $r$ around $\bbeta^*$.
The following proposition characterizes the sparsity of $\ti{\bbeta}^1$, which is also a deterministic result.

\begin{prop}
Suppose that Condition \ref{cond:lse} holds. If $\lambda,\varepsilon_1$ and $r$ satisfy \eqref{fan1}, then $\norm{\tilde{\bbeta}^1_{S^c}}_0\leq\tilde{s}$, i.e. $\tilde{\bbeta}^1$ is $s+\tilde{s}$ sparse. Here $\ti{s}$ is as described in Condition \ref{cond:lse}.
\label{prop:stage1sparsity}
\end{prop}

To determine  $\lambda$, we prove a tail probabilistic bound for $\norm{\nabla\cL(\bbeta^*)}_{\infty}$ as well $\norm{\nabla\cL(\bbeta^*)_S}_{2}$.

\begin{prop}
Suppose that $\mathbb{P}\left\{\sup\limits_t\norm{\bX_i(t)}_{\infty}\leq M\right\}=1$,
 then
\[
\mathbb{P}\left(\norm{\nabla\cL(\bbeta^*)}_{\infty}\geq 2M\sqrt{\frac{2\log(2p/\varepsilon_0)}{n}}\right)\leq \varepsilon_0,\quad \mathbb{P}\left(\norm{\nabla\cL(\bbeta^*)_S}_{2}\geq 2M\sqrt{\frac{\left(\sqrt{-64\log \varepsilon_0}+1\right)s}{n}}\right)\leq \varepsilon_0.
\]
\label{prop:maxgrad}
\end{prop}
\begin{remark}
From Proposition \ref{prop:maxgrad}, we know that with high probability $\norm{\nabla\cL(\bbeta^*)}_{\infty}$ in Cox's model is at the order of $\sqrt{\log p/n}$. By taking $\lambda=C_2\sqrt{\log p/n}$ with a large enough $C_2>0$ and taking $\varepsilon_1\lesssim\sqrt{\log p/n}$, when the radius $r$ in the LSE condition satisfies $r\geq36\rho_*^{-1}\lambda s$, the condition in Proposition \ref{prop:contraction} and Proposition \ref{prop:stage1sparsity} holds.
\end{remark}

Proposition \ref{prop:contraction} and Proposition \ref{prop:stage1sparsity} together suggest that $\tilde{\bbeta}^1$ falls in  $\cC(s+\tilde{s},r)$. This justifies the validity of directly optimizing the original nonconvex loss function (\ref{eqloss}) when starting from the warm initialization $\tilde{\bbeta}^1$ in the second stage.


\begin{prop}\label{prop:decomposition2}
Suppose that Conditions \ref{cond:lse} and \ref{cond:concave} hold. Let $\cE_1=\{j: \vert\beta^*_j\vert\leq a_1\lambda\}$ with $a_1$ defined in Condition \ref{cond:concave}. If $2\left(\norm{\nabla\cL(\bbeta^*)}_{\infty}+\varepsilon_2\right)\leq\lambda\lesssim r/s$, then $\tilde{\bbeta}^2$ satisfies the following inequality
\[
\norm{\tilde{\bbeta}^2-\bbeta^*}_2\leq C_3\left(\norm{\nabla\cL(\bbeta^*)_S}_2+\varepsilon_2\sqrt{s}+\lambda\sqrt{\abs{\cE_1\cap S}}\right)
\]
where $C_3>0$ is a constant.
\end{prop}

 Proposition \ref{prop:decomposition2} shows that the estimation error is upper bounded by the oracle rate, the optimization error and a bias term introduced by regularization. The bias only exists on $\cE_1\cap S$ where the signal strength is not strong enough. Based on this decomposition, we arrive at  the following result, showing that under a condition on the true signal $\bbeta^*$, the obtained estimator can achieve the weak oracle property.

\begin{thm}[Weak Oracle Property]\label{thm:weak}
Suppose that Conditions \ref{cond:lse} and \ref{cond:concave} hold and $\norm{\bbeta^*_S}_{\min}\geq a_1\lambda$. Take $\lambda=4M\sqrt{2\log(2p/\varepsilon_0)/n}$ where $M=\max\limits_{1\leq i\leq n}\sup\limits_{t\geq 0}\norm{\bX_i(t)}_{\infty}$ as in Proposition \ref{prop:maxgrad}, $\varepsilon_0>0$ and $\varepsilon_2\lesssim\sqrt{1/n}$, then
\begin{gather*}
\norm{\tilde{\bbeta}^2-\bbeta^*}_2\leq C_4\sqrt{\frac{s}{n}}
\end{gather*}
for some constant $C_4>0$ with probability at least $1-\varepsilon_0$.
\label{thm:weakoracle}
\end{thm}


Before establishing the strong oracle property of $\hat{\bbeta}^2$, we define the oracle estimator $\hat{\bbeta}^0$ to be
\begin{equation}
\hat{\bbeta}^0=\argmin\limits_{\text{supp}(\bbeta)=S}\cL(\bbeta).
\end{equation}

\begin{thm}[Strong Oracle Property] 
Suppose that Conditions \ref{cond:lse} and \ref{cond:concave} hold, $\norm{\hbbeta^0-\bbeta^*}_{\max}\leq\eta_n\lesssim\lambda$
and $\norm{\bbeta_S^*}_{\min}\geq a_1\lambda+\eta_n$. Let $\hat{\bbeta}^2$ be the exact solution of the second stage, i.e. $\varepsilon_2=0$. Take $\lambda=2K\sqrt{2\log(2p/\varepsilon_0)/n}$ with some small $\varepsilon_0>0$, we have with probability greater than $1-\varepsilon_0$
\begin{gather*}
\hbbeta^2=\hbbeta^0.
\end{gather*}
\label{thm:strongoracle}
\end{thm}

\begin{rmk}
$\norm{\hbbeta^0-\bbeta^*}_{\max}\leq\eta_n\lesssim\lambda$ is a mild condition since the order of the left hand side is related to the intrinsic dimension $s$ while $\lambda\asymp\sqrt{\log p/n}$ grows with $p$. Theorem \ref{thm:strongoracle} suggests the strong oracle property of the exact solution of the second stage. 
\end{rmk}

\subsection{Computational Theory}
In this section, we study the computational complexity of TLAMM in terms of the number of iterations needed in each stage. We need an additional Lipschitz condition on the gradient of Cox's  loss function.

\begin{cond}[Lipschitz Condition]\label{cond:lip}
$\norm{\nabla\cL(\bbeta_1)-\nabla\cL(\bbeta_2)}_2\leq\rho_c\norm{\bbeta_1-\bbeta_2}_2$, for $\bbeta_1, \bbeta_2\in B_2(\bbeta^*,R/2)$, where $\rho_c$ is a constant and $R\lesssim\norm{\bbeta^*}_2+\lambda\sqrt{s}$.
\end{cond}


\begin{thm}\label{thm:tscomplexity}
Suppose that Conditions \ref{cond:lse} and \ref{cond:lip} hold and take $\lambda\asymp\sqrt{\log p/n}$. With probability at least $1-p^{-1}$, we need at most $(1+\gamma_u)^2R^2\rho_c^2/\varepsilon_1^2$ LAMM iterations in the first stage to reach $\varepsilon_1$-optimal solution and at most $C_5\log(C_6\sqrt{s\log p/n}/\varepsilon_2)$ LAMM iterations in the second stage to reach $\varepsilon_2$-optimal solution, where $C_5>0$ and $C_6>0$ are two constants.
\end{thm}

The sublinear rate in the first stage is due to the lack of strong convexity of the loss function, since TLAMM could tolerate an arbitrarily bad initialization. Once we enter the second stage, the strong convexity of the loss function allows the algorithm to admit a geometric convergence rate.



\cite{fan2018lamm} approximate the nonconvex loss function using a series of local linear approximations. In each stage, they solve an adaptive Lasso problem using the LAMM algorithm, and then update the tuning parameter using the gradient of the nonconvex penalty function at the latest solution  for finer approximation in the next stage. 
The personalized tuning parameter for each entry gradually eliminates the shrinkage bias caused by the Lasso penalty. As a result, their algorithms requires $T\asymp\log\log p$ stages after the first contraction stage to achieve good approximation and a total number of $(1+\gamma_u)^2R^2\rho_c^2/\varepsilon_1^2+C_5\log\log p\cdot\log(C_6\sqrt{s\log p/n}/\varepsilon_2)$ LAMM iterations to complete the algorithm.

In comparison, TLAMM replaces the $T$ approximation stages with one single stage that directly optimizes the nonconvex problem.  As a consequence, the number of LAMM iterations needed is reduced to $(1+\gamma_u)^2R^2\rho_c^2/\varepsilon_1^2+C_5\log(C_6\sqrt{s\log p/n}/\varepsilon_2)$. This  reduces the computational complexity while maintaining all the good statistical  properties of the original I-LAMM algorithm.

\section{Numerical Results}\label{sec:simu}

In this section, we use numerical experiments to examine the finite-sample performance of our proposed algorithm. We simulate 100 datasets with $n=100, 200, 300, 400$ and $p=20, 50, 100, 200, 400, 800$ in the proportional hazards model 
\begin{equation}
\lambda(t\vert\bx)=\exp(\bbeta^{*\top}\bx),
\end{equation}
where $\bbeta^*\in\mathbb{R}^p$ is a sparse vector with $s_0=\norm{\bbeta^*}_0=10$. For each sample, the censoring time follows an exponential distribution with mean $U\exp(\bbeta^{*\top}\bx)$ where $U$ is a uniformly distributed variable in $[2,3]$. We conduct the experiment in the following settings:

\begin{enumerate}
\item Full simulation: We set all the non-zero entries of $\bbeta_0$ to be 0.8. The covariates $\bx$ are generated from $\cN({\bf 0},\bSigma)$, where $\bSigma$ is a correlation matrix $\bSigma=(\rho_{ij})$ in these forms:
\begin{enumerate}
\item Independent design where $\rho_{ij}=\mathds{1}_{\{i=j\}}$;
\item Constant correlation design with $\rho_{ij}=0.5$ for $i\neq j$ and $\rho_{ij}=1$ for $i=j$;
\item Autoregressive correlation design with $\rho_{ij}=0.95^{\abs{i-j}}$.
\end{enumerate}

With large enough signal, our goal is to show the oracle property of the T-LAMM estimators.

\item Semi-simulation: We set the non-zero entries of $\bbeta_0$ to be $1, 0.9, 0.8, ..., 0.1$. The covariates in this case are sampled from skin cutaneous melanoma (SKCM) dataset in The Cancer Genome Atlas (TCGA, \url{http://cancergenome.nih.gov/}). Detailed introduction of this dataset can be found in Section \ref{sec:data}.

With diversified beta, oracle property is no longer guaranteed. We look to illustrate the variable selection power of TLAMM.
\end{enumerate}

In all the numerical experiments, we take the nonconvex penalty functions to be SCAD and MCP. The $L_1$ penalty is also included in the numerical experiments for comparison purpose. We take $\lambda=c\sqrt{\log p/n}$, in which $c$ is tuned from $0.05\times\{1,2,...,20\}$ at $n=200$ and $p=100$ by 3-fold cross validation.

\subsection{TLAMM with fully simulated data}

In this section, all non-zero entries of $\bbeta_0$ are fixed at 0.8. The covariates $x$ are generated from three different designs of covariance matrices: independent, constant correlation and autoregressive correlation. In this setting, the average censored rate of the samples in 100 repetitions is around 55\%, i.e. about 55\% of samples have already failed at the time of censoring. This number varies with different sample sizes, dimensions and correlation structures in covariates, but we see it lying between 52\% and 58\% in all cases. 

\subsubsection{Weak Oracle}

To verify the proposed weak oracle property, we record the $L_2$ estimation error under different dimension settings and with various penalties. Figure \ref{fig:ts-weakoracle1}, Figure \ref{fig:ts-weakoracle2} and Figure \ref{fig:ts-weakoracle3} record the median results in all three cases.

\begin{figure}[htb!]
\includegraphics[scale=0.45]{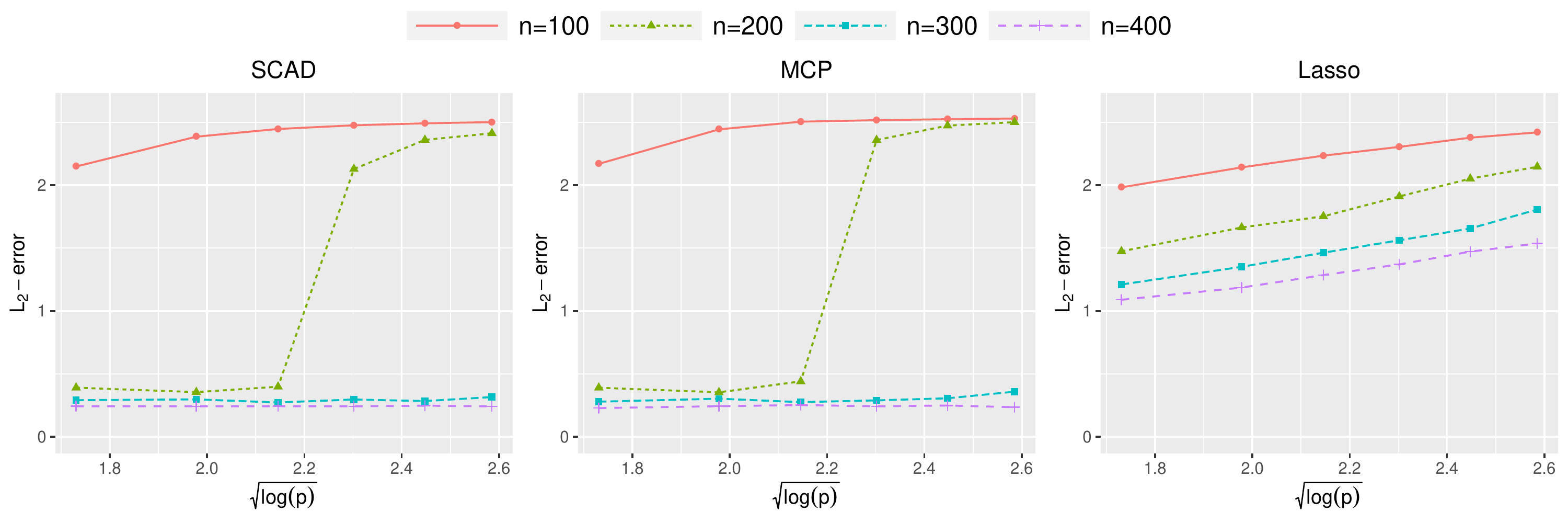}
\caption{$\norm{\hbbeta-\bbeta^*}_2$ versus $\sqrt{\log p}$ when $n=$100, 200, 300 and 400 under independent design.}
\label{fig:ts-weakoracle1}
\end{figure}

\begin{figure}[htb!]
\includegraphics[scale=0.45]{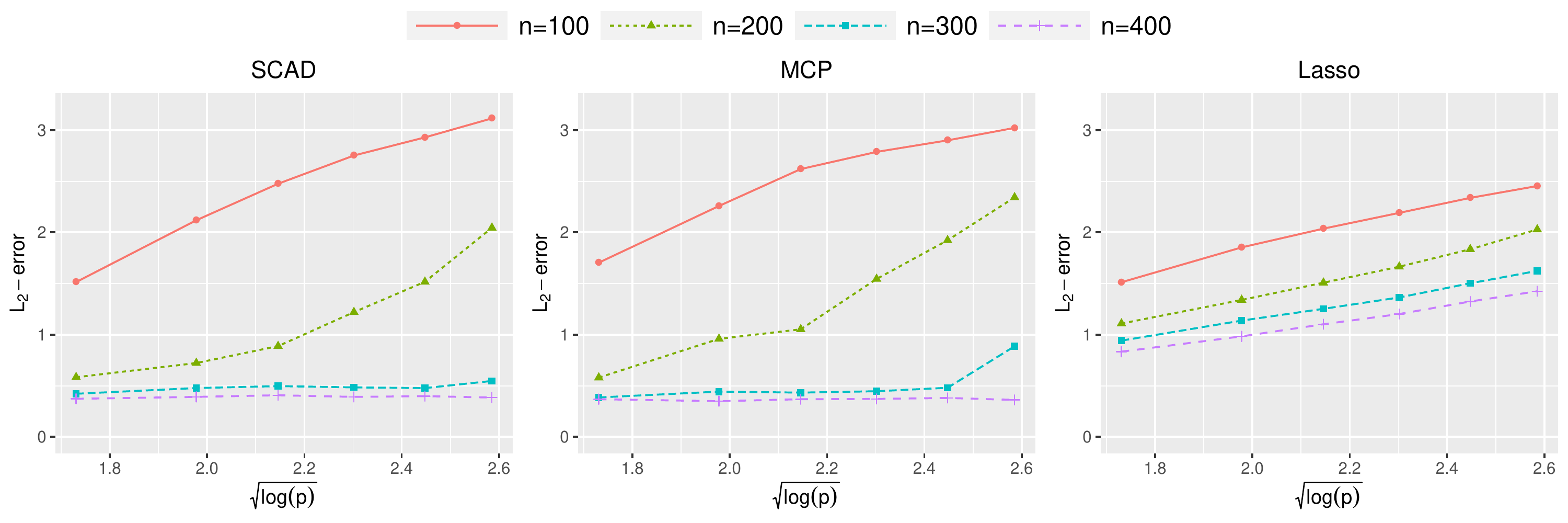}
\caption{$\norm{\hbbeta-\bbeta^*}_2$ versus $\sqrt{\log p}$ when $n=$100, 200, 300 and 400 under constant correlation design.}
\label{fig:ts-weakoracle2}
\end{figure}

\begin{figure}[htb!]
\includegraphics[scale=0.45]{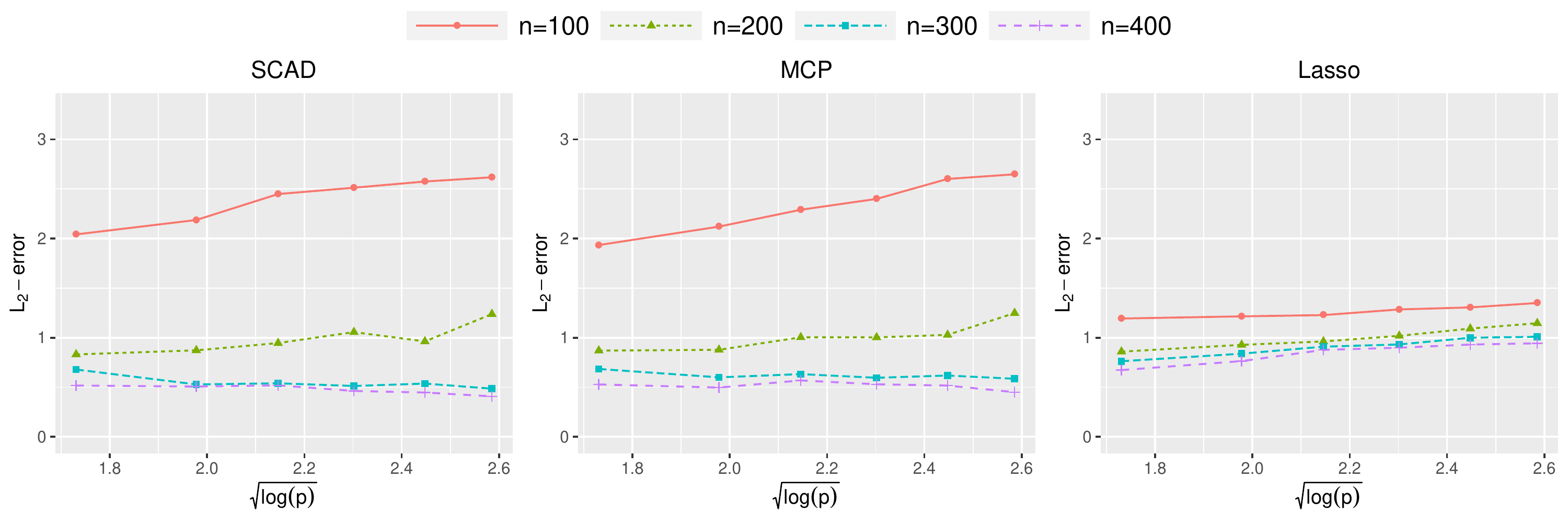}
\caption{$\norm{\hbbeta-\bbeta^*}_2$ versus $\sqrt{\log p}$ when $n=$100, 200, 300 and 400 under autoregressive correlation design.}
\label{fig:ts-weakoracle3}
\end{figure}

 Overall, TLAMM performs similarly under different designs. With nonconvex penalty functions, according to Theorem \ref{thm:weakoracle}, the estimation error is in the order of $\sqrt{s\log s/n}$ and does not grow with $p$ under the minimal signal condition. In the figures, when using SCAD and MCP, the estimation error remains constant as $p$ grows for fixed and large enough $n$. This supports the weak oracle property we proved in Theorem \ref{thm:weak}. When the signal strength is not large enough comparing with the noise, that is when $n$ is small, the oracle rate is no longer achievable by TLAMM and the estimation error grows linearly with $\sqrt{\log p}$. This is best illustrated under the independent design.

Meanwhile, when using Lasso penalty, the oracle rate can not be achieved and the error rate is on the order of $\sqrt{s\log p/n}$. Therefore, for fixed $n$, the slope of the error is $\sqrt{s/n}$ as $\sqrt{\log p}$ grows. This is supported by our simulation results by observing that the estimation error grows linearly with $\sqrt{\log p}$. 

\subsubsection{Variable Selection Property}

To verify the selection property of TLAMM for Cox's proportional hazards regression model, we study the accuracy of variable selection with different dimensionalities, correlation designs and penalties. The true positive rate (also known as sensitivity) and the true negative rate (also known as specificity) are recorded in each repetition. Figure \ref{fig:ts-strongoracle1}, Figure \ref{fig:ts-strongoracle2} and Figure \ref{fig:ts-strongoracle3} summarize the median results.
\begin{figure}[htbp]
\includegraphics[scale=0.45]{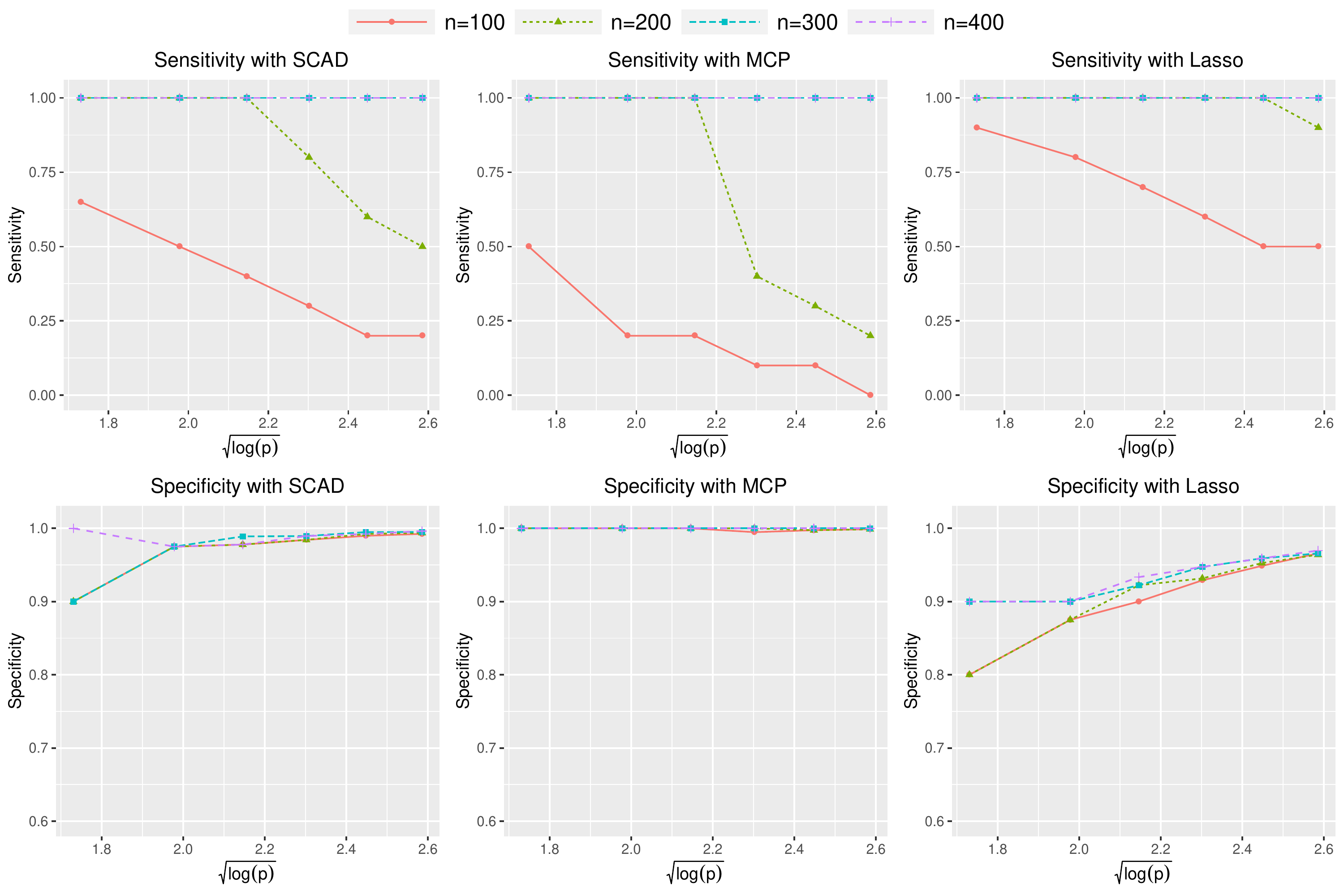}
\caption{Sensitivity and specificity with SCAD, MCP and Lasso under independent design.}
\label{fig:ts-strongoracle1}
\end{figure}

\begin{figure}[htbp]
\includegraphics[scale=0.45]{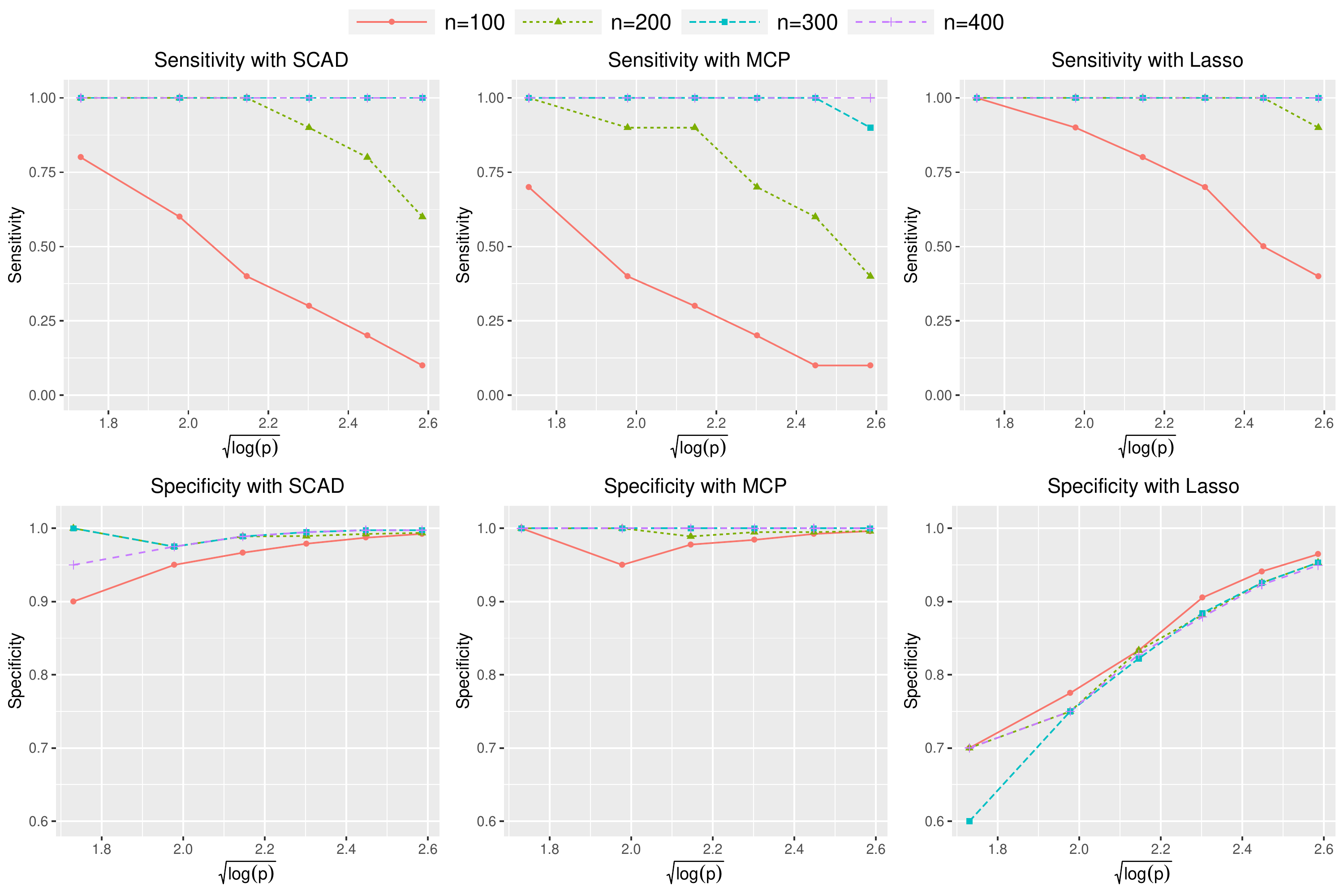}
\caption{Sensitivity and specificity with SCAD, MCP and Lasso under constant correlation design.}
\label{fig:ts-strongoracle2}
\end{figure}

\begin{figure}[htbp]
\includegraphics[scale=0.45]{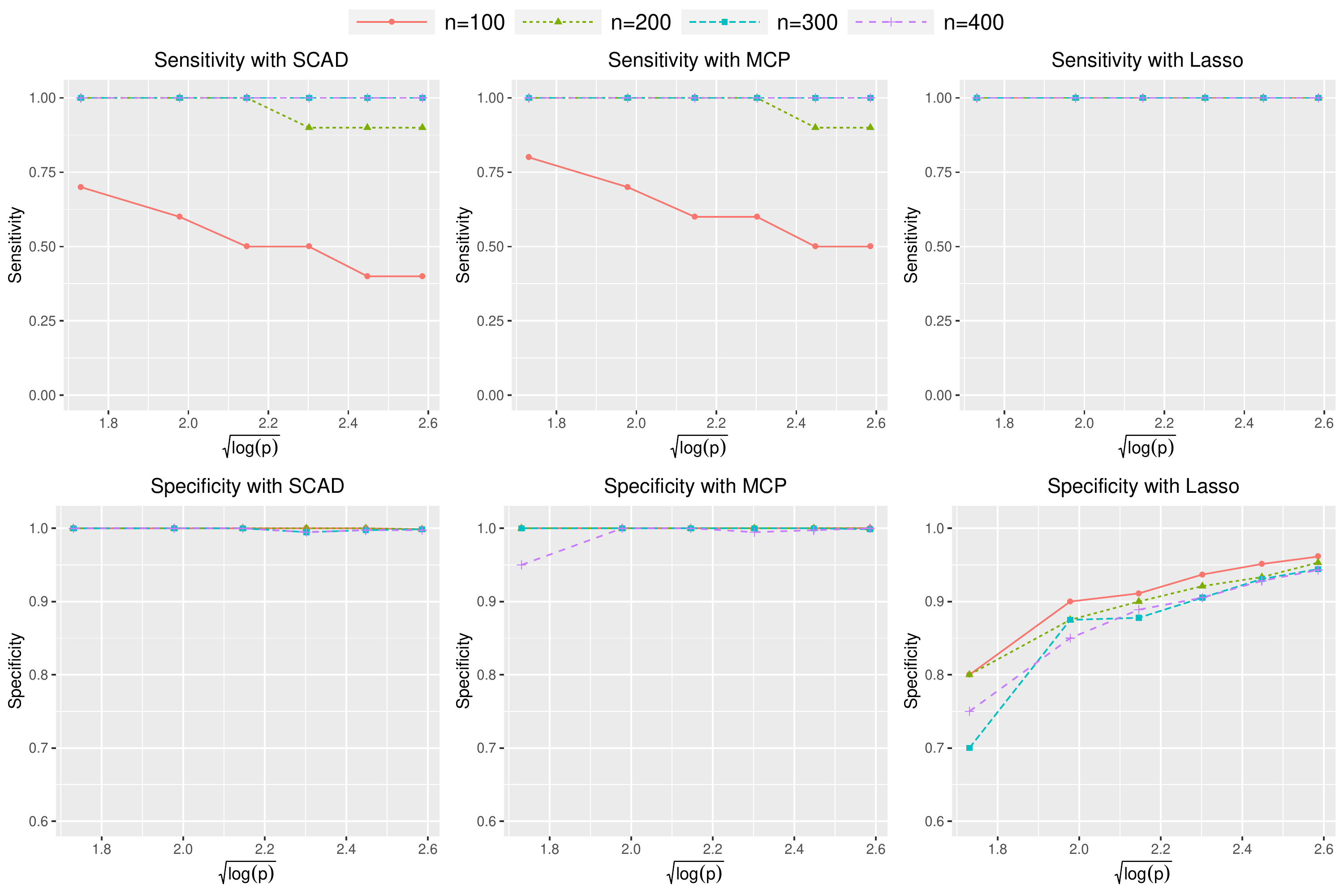}
\caption{Sensitivity and specificity with SCAD, MCP and Lasso under autoregressive correlation design.}
\label{fig:ts-strongoracle3}
\end{figure}

These results show that, when using nonconvex regularizers like SCAD and MCP, both the sensitivity and specificity are 1 or very close to 1 when the sample size is relatively large. Lasso tends to over select and result in undesirable performances in terms of specificity.  TLAMM performs stably in  all three cases.

\subsubsection{Comparison with Other Algorithms}
In this section, we compare the performances of TLAMM with those of  I-LAMM and  other stage-of-the-art algorithms. We also compare the computational complexity of TLAMM and I-LAMM implied by the computational time. Again, Lasso, SCAD and MCP are used. The Lasso estimator and the post-Lasso estimator \citep{belloni2013least} are computed using the R package \texttt{ncvreg} and \texttt{survival}. The SCAD estimators and MCP estimators are computed using R package \texttt{ncvreg}, the I-LAMM algorithms in \cite{fan2018lamm} and our TLAMM algorithm. We also report the performances of the oracle estimator as benchmarks.


All the results in this section are computed  under the independent design. The constant $c$ in $\lambda=c\sqrt{\log p/n}$ is tuned by the $3$-fold cross validation at $n=200$ and $p=100$. The number of nonzero coefficients is fixed at $10$ as the sample size and the dimensionality grow. Each experiment is repeated $100$ times.

Table \ref{tab:cmpr} collects the median $L_2$ estimation error, TP and FP 
when $n=300$ and $p=2400$. Column TP stands for the number of true discoveries and column FP indicates the number of false discoveries. Using the same nonconvex penalty functions, TLAMM and I-LAMM both outperform the coordinate descent algorithms used by \texttt{ncvreg} in terms of the estimation error and the accuracy of variable selection. TLAMM ranks the top in terms of selection accuracy.

\begin{table}[htb!]
\center
\begin{tabular}{ |c|c|c|c| }
 \hline
 Method & $L_2$ error & TP & FP \\ \hline
  Oracle & 0.29 & 10 & 0 \\
  Lasso (ncvreg) & 1.40 & 10 & 116 \\
   Refit (ncvreg) & 10.20 & 10 & 116 \\
         MCP (ncvreg)& 0.67 & 10 & 3 \\
\rowcolor{Gray}
     MCP (TLAMM) & 0.34 & 10 & 0 \\
MCP (I-LAMM) &  0.39 & 10  &  1 \\
  SCAD (ncvreg) & 2.04 & 10 & 14 \\
\rowcolor{Gray}
 SCAD (TLAMM) & 0.36 & 10 & 7 \\
 SCAD (I-LAMM) & 0.32 & 10 & 10 \\

 \hline
\end{tabular}
\caption{The median $L_2$ error, TP and FP of various estimators when n=300 and p=2400.}
\label{tab:cmpr}
\end{table}

Figure \ref{fig:cmpr} shows that TLAMM and I-LAMM outperform the traditional methods in terms of $L_2$ error especially when $p$ is large and coordinate descent algorithm starts to break down. 
This demonstrates the stability of TLAMM and I-LAMM, which is also guaranteed by the theory. 

\begin{figure}[htb!]
\centering
\includegraphics[scale=0.55]{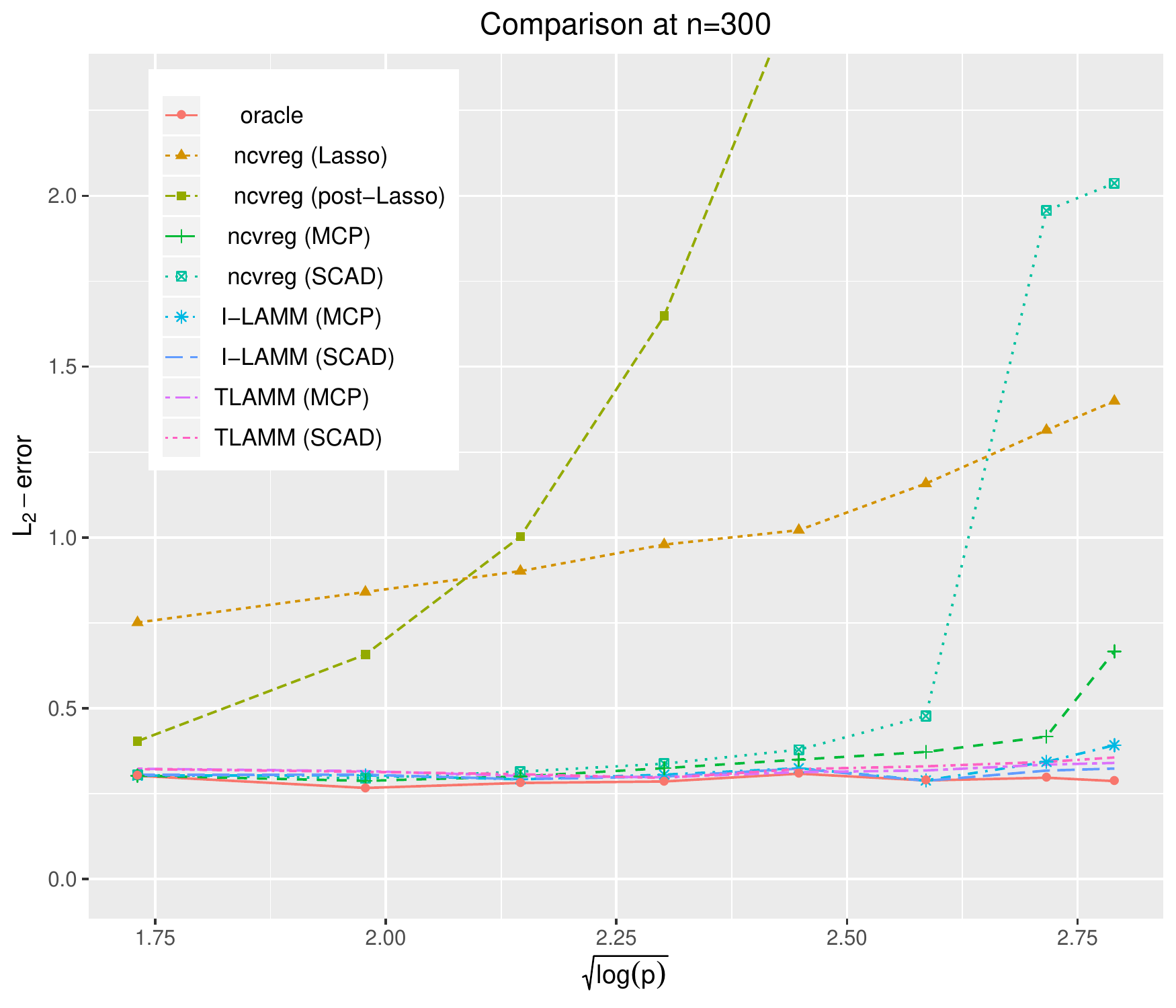}
\caption{Median $L_2$ error with fixed sample size and various dimensions. TLAMM and I-LAMM estimators achieve oracle error rate while the $L_2$ error of the traditional estimators grow with dimension.}
\label{fig:cmpr}
\end{figure}

It is also supported by the numerical experiment that TLAMM reduces the computational complexity of I-LAMM. In Figure \ref{fig:time}, we plot the average running time of each repetition using the two algorithms with the same hyper-parameters and cpus. The dimension of the problem is still $n=300$. We set $\gamma_u=2$, $\phi_0=0.1$ and $\varepsilon_c=\varepsilon_t=0.002$. For I-LAMM, we carry an extra parameter of maximum number of tightening subproblems allowed and it is set to 20. Results show that TLAMM in general takes about 30\% less computational time than I-LAMM. 

\begin{figure}[htb!]
\centering
\includegraphics[scale=0.65]{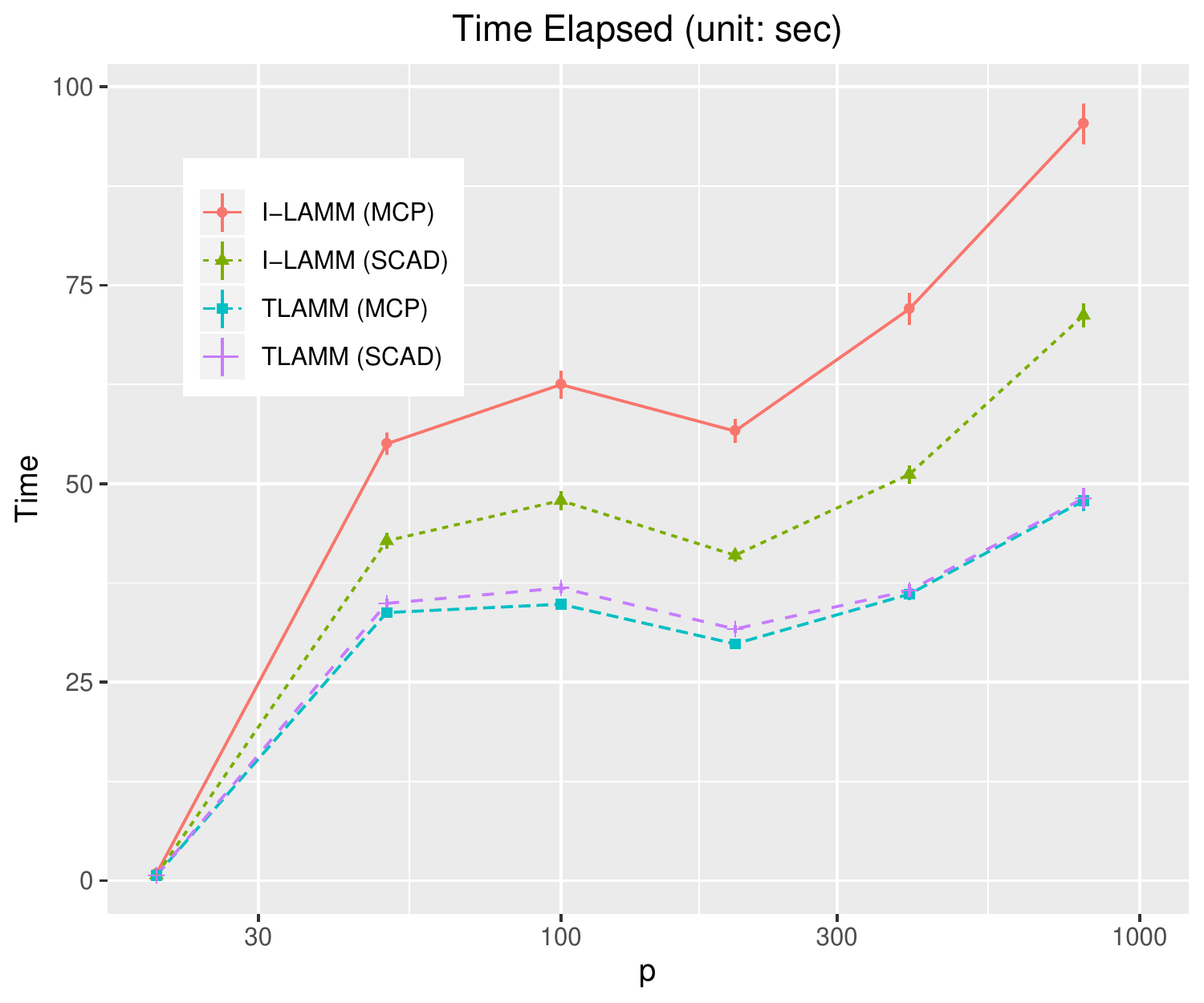}
\caption{Average computational time of each repetition at n=300 with error bars showing one standard deviation. TLAMM takes about 30\% less computational time than I-LAMM.}
\label{fig:time}
\end{figure}

\subsection{TLAMM with semi-simulated data}

In this section, the non-zero entries of $\bbeta_0$ are diversified using $1, 0.9, 0.8, ..., 0.1$. The covariates $\bx$ are empirical: they are sampled from the skin cutaneous melanoma (SKCM) dataset in The Cancer Genome Atlas. In this setting, the average censored rate of the samples is around 70\%.

As showed in Figure \ref{fig:ts-diverr} and Figure \ref{fig:ts-divselect}, when the magnitudes of the true non-zero coefficients vary and decay, the weak oracle property is no longer guaranteed. The $L_2$ estimation error grows with dimension $p$. Meanwhile, the sensitivity also decreases as $p$ increases. However, the specificity is still very high.

To better understand the variable selection accuracy of the diversified coefficients, we check the rate of being selected by magnitude at $n=400$ in Figure \ref{fig:ts-divselectrate}. While sensitivity decreases as $p$ grows, TLAMM is able to consistently select the variables with sufficiently large coefficients. This again verifies the proposed theorems in Section \ref{sec:twostage}.

\begin{figure}[htbp]
\includegraphics[scale=0.45]{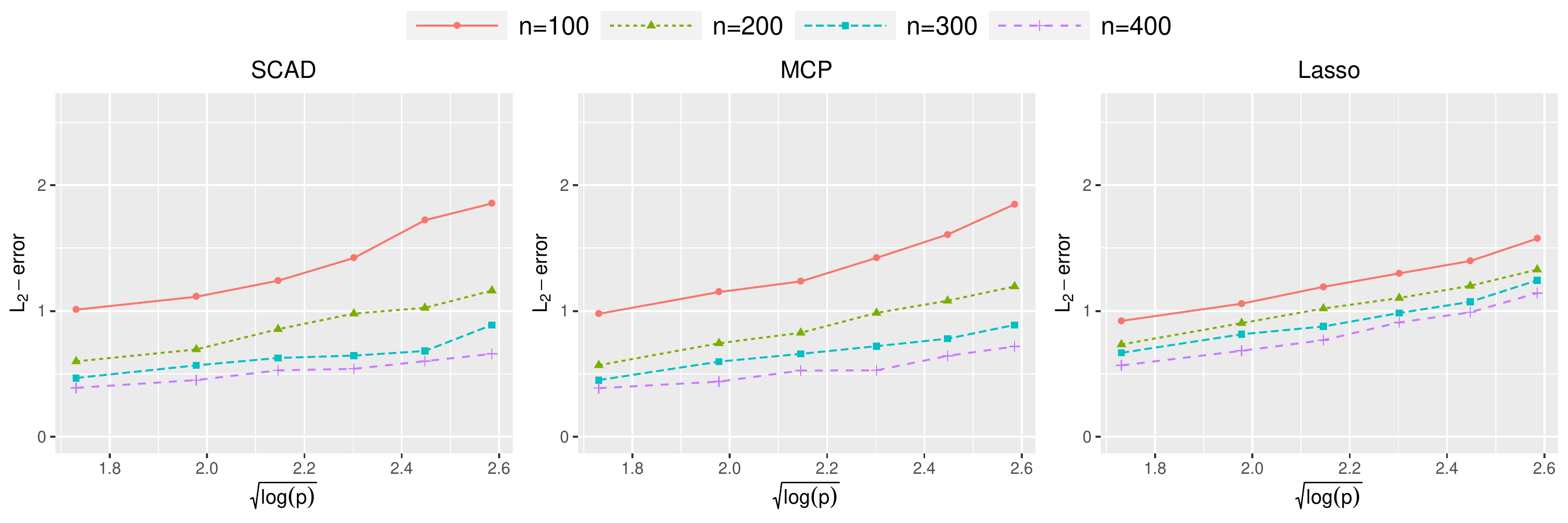}
\caption{$\norm{\hbbeta-\bbeta^*}_2$ versus $\sqrt{\log p}$ when $n=$100, 200, 300 and 400 with empirical covariates.}
\label{fig:ts-diverr}
\end{figure}

\begin{figure}[htbp]
\includegraphics[scale=0.45]{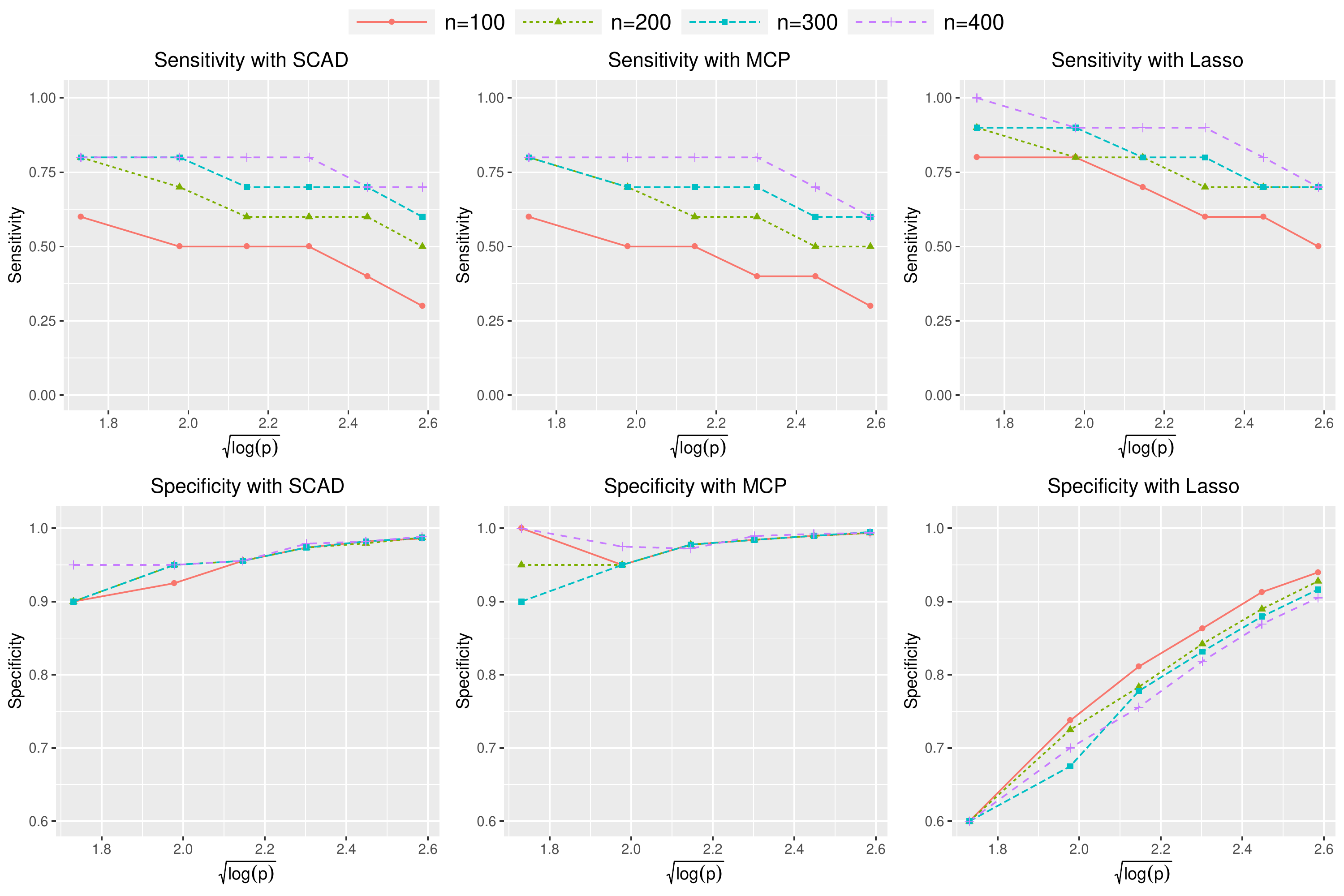}
\caption{Sensitivity and specificity with SCAD, MCP and Lasso with empirical covariates.}
\label{fig:ts-divselect}
\end{figure}

\begin{figure}[htbp]
\includegraphics[scale=0.45]{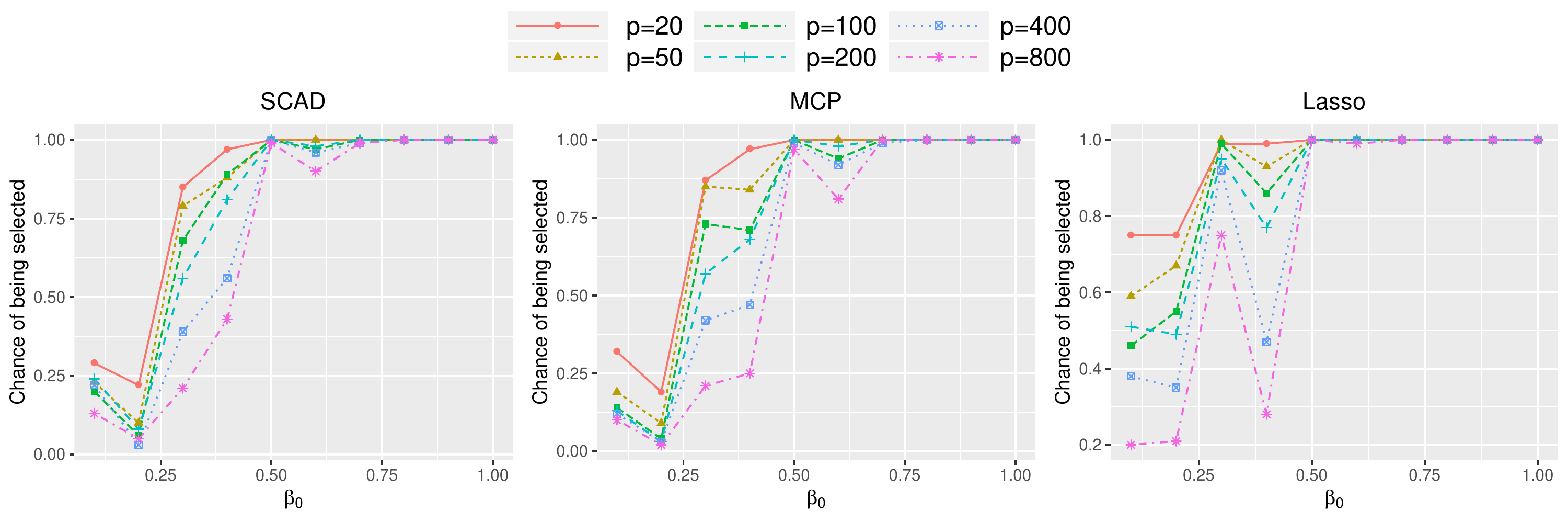}
\caption{The average selection rate of each coefficient at $n=400$.}
\label{fig:ts-divselectrate}
\end{figure}

\section{Real Data Analysis}\label{sec:data}

We apply TLAMM to the skin cutaneous melanoma (SKCM) dataset in The Cancer Genome Atlas (TCGA, \url{http://
cancergenome.nih.gov/}) . TCGA provides comprehensive profiling data on more than thirty cancer types, allowing researchers to study the roles of genes in various cancers. The dataset we used in the analysis is downloaded from UCSC Xena (\url{https://xena.ucsc.edu/public/}). 
We acquire $20,531$ mRNA expression and clinical data on a total of $461$ patients, with $155$ observed failures. For each patient, the censoring time and the censoring results are also included in the dataset.

The purpose of our study is to select the genes that are strongly associated with the survival time of melanoma patients. After removing the genes and patients with missing values, we end up having 12207 mRNA expression of 457 patients. We pre-select the top 20 genes associated with cutaneous melanoma according to a meta-analyses of over 145 papers \citep{chatzinasiou2011comprehensive}. A list of these genes could be found at \url{http://bioinformatics.cing.ac.cy/MelGene/}. We then screen out 1000 genes with the highest variability. These 1000 genes and the 20 pre-selected genes are combined together into a refined pool of 1017 
genes (since there are 3 overlapping genes). Besides the mRNA expressions, the information of gender and age are also included into the model. Our study finally utilize 457 observations and each observation has 1019 features. 

We use TLAMM to fit the regularized Cox's model to SKCM dataset using SCAD  and MCP. The tuning parameters are selected by 3-fold cross validation. The selection results are shown in Table \ref{tab:tcga}. The variables selected with the two penalty functions are almost the same: MCP selects 6 variables and SCAD selects 5. Age appears in both selected models, which is intuitive since younger patients usually have larger chances in defeating diseases. None of the 20 pre-selected genes are selected by the model.

We shall point out that a major difference between our analysis and the studies included in the meta-analysis \citep{chatzinasiou2011comprehensive} is that most of the medical research focus on marginal effect of an individual gene while we estimate the effects of gene expressions jointly, resulting to different results. 

\begin{table}[htbp]
\centering
{\renewcommand\arraystretch{1.25}
\begin{tabular}{|c|c|c|}
\hline
{\bf Penalty} & \multicolumn{2}{c|} {\bf Selected Variables} \\ \hline \hline
SCAD & \multicolumn{2}{p{8cm}|}{\raggedright Age, BOK, HSPB6,  ATF7IP2, GBP2} \\ \hline
MCP & \multicolumn{2}{p{8cm}|}{\raggedright Age, BOK, HSPB6, MBP, ATF7IP2, GBP2} \\ \hline
\end{tabular}}
\caption{Genes selected by TLAMM using skin cutaneous melanoma dataset. The genes from pre-selected pool are bolded. The genes being selected in both models are underlined.}\label{tab:tcga}
\end{table}

We also study the predictability of the fitted TLAMM models. We randomly divide 457 patient into training set and testing set, and use the model fitted on the training data to predict the hazards of the samples in the testing set. This is repeated for 20 times and the prediction results are recorded. Please notice that we do not re-calibrate $\lambda$ for each random split. Instead, we continue to use the same $\lambda$ we used in variable selection using the full sample data. 

We use the percentage of ``concordant pairs" (concordance index) to measure how well they predict. We list all possible pairs of the patients, and for each pair we decide whether they are concordant, discordant or indeterminate by comparing the prediction results and their outcomes. A pair is concordant, if the object we predict to be of lower risk survives longer. Meanwhile, a pair is discordant if the object we predict to be of higher risk survives longer. We remind the readers that not all pairs fall into these two categories. More often the pairs are indeterminate, meaning that we predict the risk of the two objects but we do not know which fails first. The concordance index is defined as \[\text{concordance index} = \frac{\text{number of concordant pairs}}{\text{number of concordant pairs} + \text{number of discordant pairs}}.\]

\begin{table}[htbp]
\centering
\begin{tabular}{|c|c|c|c|}
\hline
 & SCAD & MCP & LASSO \\ \hline
TLAMM & 0.612 & 0.602 & 0.618 \\ 
I-LAMM &  0.612 & 0.597 & 0.618 \\ 
ncvreg & 0.621 & 0.614 & 0.617 \\ \hline
 \end{tabular}
\caption{Average concordance index using TLAMM, I-LAMM and ncvreg and with different penalty functions in 20 repetitions.}\label{tab:tcga}
\end{table}

The concordance index is around 0.6 for all methods, suggesting the good predictability of the fitted model using both TLAMM and I-LAMM.

\section{Discussion}\label{sec:dis}

In this paper, we propose a two-stage algorithmic approach called TLAMM to fit the  Cox's proportional hazards model in high dimensions. The proposed algorithm achieves  both statistical and computational guarantees: we show that  the complexity of TLAMM is well controlled while the consistency of the estimator and the accuracy of the selection are guaranteed. Moreover, TLAMM has the potential to be extended to other common models such as linear regression and logistic regression, as long as the loss function shares a similar geometry locally around the true signal.

\bibliographystyle{ims}
\bibliography{ref}

\newpage
\appendix

We present the proof of Theorem \ref{thm:minse}, Theorem \ref{thm:strongoracle}, Theorem \ref{thm:tscomplexity}, Proposition \ref{prop:contraction}, Proposition \ref{prop:maxgrad} and Proposition \ref{prop:decomposition2} in the appendix. The proof of Proposition \ref{prop:stage1sparsity} could be found in \cite{fan2018lamm}.

We first introduce some notations we use in throughout the appendix. We define the shifted loss function $\ti{\cL}(\bbeta)=\cL(\bbeta)+p_{\lambda}(\bbeta)-\lambda\norm{\bbeta}_1$, with which the penalized loss function $\cL(\bbeta)+p_{\lambda}(\bbeta)$ could be written as $\ti{\cL}(\bbeta)+\lambda\norm{\bbeta}_1$, and is also denoted as $F(\bbeta,\lambda)$ or $F(\bbeta)$ for simplicity. In the second stage, we write $\Psi_{\ti{\cL},\lambda,\phi}(\bbeta_1,\bbeta_2)=\Psi_{\ti{\cL},\phi}(\bbeta_1,\bbeta_2)+\lambda\norm{\bbeta_1}_1$ as the penalized majorize function and omit $\ti{\cL}$ in the subscript when there is no ambiguity. By carefully selecting the parameter in the penalty function, the strong convexity of $\cL(\cdot)$ within local $\ell_1$ cone is preserved with $\ti{\cL}(\cdot)$. When there is no ambiguity, we drop the dependence on sparsity level and the radius and write $\rho_+(2s+2\ti{s},r)$, $\rho_-(2s+2\ti{s},r)$, $\kappa_+(2s+2\ti{s},r)$ and $\kappa_-(2s+2\ti{s},r)$ as $\rho_+$, $\rho_-$, $\kappa_+$ and $\kappa_-$. 

The proof presented here are mainly about the second stage estimators. For each LAMM estimator within the second stage, we omit the stage number in the superscript and write $\bbeta^{2,k}$ as $\bbeta^k$ and write $\phi^{\ell,k}$ as $\phi$ for simplicity when there is no ambiguity.

\begin{section}{Technical Lemmas}
All the lemmas collected here are for the second stage if without specification.

\begin{lemma}[$\ell_1$ Cone Property For Approximate Solution in the Second Stage]\label{lem:akkt1}
If $\|\nabla\cL(\bttc)\|_\infty+\varepsilon\leq \lambda$ and $\norm{\bbeta^*}_{\min}\geq a_1\lambda$, we must have
	\begin{equation}
	\|(\tbt-\bttc)_{S^c}\|_1\leq \frac{\|\nabla\cL(\bttc)\|_\infty+\varepsilon}{\lambda-(\|\nabla\cL(\bttc)\|_\infty+\varepsilon)}\|{(\tbt-\bttc)}_{S}\|_1,
	\label{eq:cone}
	\end{equation}
where $\tilde{\bbeta}$ is a stage estimator.
\end{lemma}

\begin{proof}[Proof of Lemma \ref{lem:akkt1}]
Lemma \ref{lem:akkt1} depicts the $\ell_1$ cone property for the approximate solution in the second stage.

For any $\ti{\bxi}\in\partial \|\tbt\|_1$, let  $\mbu=\nabla\cL(\tbt)+p'_{\lambda}(\tbt)=\nabla\tilde{\cL}(\tbt)+\lambda\ti{\bxi}$. By the Mean Value theory,
 there exists a $\gamma\in[0,1]$, such that $\nabla\tilde{\cL}(\tilde\btt)-\nabla\tilde{\cL}(\bttc)=\big[\nabla^2\tilde{\cL}\big(\gamma\bttc+(1-\gamma)\tbt\big)\big]\big(\tbt-\bttc\big)$. Write $\Hb=\nabla^2\tilde{\cL}\big(\gamma\bttc+(1-\gamma)\tbt\big)$. Then we have
	\begin{align*}
	\big\langle\nabla\tilde{\cL}(\tilde\btt)\!+\!\lambda\tilde{\bxi}, \tilde\btt\!-\!\bttc\big\rangle&\!=\!\big\langle\nabla\tilde{\cL}(\bttc)\!+\! \bH(\ti \btt-\bttc)+\lambda\ti{\bxi}, \tilde\btt\!-\!\bttc\big\rangle.\\
	&\!\leq\! \|\mbu\|_\infty\|\tbt\!-\!\btt\|_1\notag
	\end{align*}
Given that $\norm{\bbeta^*}_{\min}\geq a_1\lambda$, $\nabla\tilde{\cL}(\bbeta^*)=\nabla\cL(\bbeta^*)-\lambda\bxi^*$ where $\bxi^*\in\partial \norm{\bbeta^*}_1$ since $p'_{\lambda}(\bbeta^*)=\bf 0$. Using  the fact  $(\tbt-\bttc)^T\Hb(\tbt-\bttc)\geq 0$, we have
\begin{equation}
0\leq \|\mbu\|_\infty\|\tbt-\bttc\|_1- \underbrace{\big\langle\nabla\cL(\bttc), \tbt-\bttc\big\rangle}_{\Rom{1}} -\underbrace{\big\langle\lambda\ti{\bxi}-\lambda\bxi^*, \tbt-\bttc\big\rangle}_{\Rom{2}}.
\label{eq:e7}
\end{equation}
Using a similar argument in the proof of Proposition \ref{prop:decomposition2}, we have
	$
	\Rom{1}\geq -\|\nabla\cL(\bttc)\|_\infty\|\tbt-\btt\|_1,
	$
and
	\begin{equation}
	\begin{split}
	\Rom{2}&=\big\langle\lambda\ti{\bxi}-\lambda\bxi^*,\tbt-\bttc\big\rangle \\
	&=\big\langle(\lambda\ti{\bxi}-\lambda\bxi^*)_{S^c},(\tbt-\bttc)_{S^c}\big\rangle+\big\langle(\lambda\ti{\bxi}-\lambda\bxi^*)_{S},(\tbt-\bttc)_{S}\big\rangle \\
	&\geq \lambda\|(\tbt-\bttc)_{S^c}\|_1.
	\end{split}
	\label{eq:a3}
	\end{equation}
In \eqref{eq:a3}, $\ti{\bxi}-\bxi^*=\ti{\bxi}$ have the same sign as $\tbt-\bttc=\tbt$ on $S^c$, making the first inner product equal to $\lambda\|(\tbt-\bttc)_{S^c}\|_1$ while $\ti{\bxi}-\bxi^*$ have the same sign as $\tbt-\bttc$ on $S$, making the second inner-product non-negative.
Plugging \eqref{eq:a3} into (\ref{eq:e7}) and taking infimum with respect to $\ti{\bxi}\in \partial \|\tbt\|_1$ yields
	\$\label{}
	0
	&\leq  -(\lambda-(\|\nabla\cL(\bttc)\|_\infty+\omega_{\lambda}(\tbt)))\|(\tbt-\bttc)_{S^c}\|_1\\
	&\hspace{1cm}+(\|\nabla\cL(\bttc)\|_\infty+\omega_{\lambda}(\tbt))\|(\tbt-\bttc)_{S}\|_1,
	\$
or equivalently	
	\$
	\|(\tbt-\bttc)_{S^c}\|_1 &\leq \frac{\|\nabla\cL(\bttc)\|_\infty+\omega_{\lambda}(\tbt)}{\lambda-(\|\nabla\cL(\bttc)\|_\infty+\omega_{\lambda}(\tbt))} \|(\tbt-\bttc)_{S}\|_1
	\$
Using the stopping criterion, i.e.  $\omega_{\lambda}(\tbt)\leq \varepsilon$, we have that
	\$
	\|(\tbt-\bttc)_{S^c}\|_1 \leq \frac{\|\nabla\cL(\bttc)\|_\infty+\varepsilon}{\lambda-(\|\nabla\cL(\bttc)\|_\infty+\varepsilon)} \|(\tbt-\bttc)_{S}\|_1
	\$
	Therefore we proved the desired result.
\end{proof}

\begin{lemma}\label{ls}
 Suppose the same conditions in Theorem \ref{thm:tscomplexity} hold.
	Assume $\btt^{k+1},\btt^{k}\in B_2(r,\bttc)$ such that  $\max\{\|\btt^{k+1}_{S^c}\|_0,\|\btt^{k}_{S^c}\|_0\}\leq\ti s$.  For the LAMM algorithm, we have
	\$
	\rho_-(2s+2\ti s,r)\leq\phi\leq \gamma_u\rho_+( 2s+2\ti s, r).
	\$
\end{lemma}

\begin{proof}[Proof of Lemma \ref{ls}]
The Lemma is borrowed from Lemma E.7 in \cite{fan2018lamm} and the proof could be found therein.
\end{proof}

\begin{lemma}\label{lem:stopping2}
	If $\btt^{k-1},\btt^{k}\in B_2(r/2,\bttc)$, $\|(\btt^{k})_{S^c}\|_0\leq \ti s$ and $\|(\btt^{k-1})_{S^c}\|_0\leq \ti s$, then for any $k\geq 1$, we have
	\$
	\omega_{\lambda} (\btt^{k})\leq (1+\gamma_u)\rho_+(2s+2\ti s,r)\|\btt^{k}-\btt^{k-1}\|_2.
	\$
\end{lemma}

\begin{lemma}\label{lem:diff2}
 We have
	\$
	F(\btt^{k},\lambda)-F(\btt^{k-1},\lambda)\leq -\frac{\phi}{2}\|\btt^{k}-\btt^{k-1}\|_2.
	\$
\end{lemma}

\begin{proof}[Proof of Lemma \ref{lem:stopping2}]
Since $\btt^{k}$ is the exact solution to the $k$th iteration in the second stage,
    the first order optimality condition holds:  there exists a $\bxi^{k}\in \partial \norm{\bbeta^k}_1$ such that
    \$\label{}
    \nabla\ti{\cL}(\btt^{k-1})+\phi(\btt^k-\btt^{k-1})+\lambda\bxi^k=0.
    \$
    \noindent   Then for any $\mbu$ such that $\|\mbu\|_1=1$, we have
    \begin{equation}
    \begin{split}
    \langle\nabla\ti{\cL}(\btt^k)\!+\!\lambda\bxi^k,\mbu\rangle
    &\!=\!\big\langle\nabla\ti{\cL}(\btt^k),\mbu\rangle\!-\!\langle\nabla\ti{\cL}(\btt^{k-1})\!+\!\phi(\btt^k\!-\!\btt^{k-1}),\mbu\big\rangle \\
    &\!=\!\big\langle\nabla\ti{\cL}(\btt^k)\!-\!\nabla\ti{\cL}(\btt^{k-1}),\mbu\big\rangle\!-\!\big\langle\phi(\btt^k\!-\!\btt^{k-1}),\mbu\rangle\\
    &\!\leq\! \|\nabla\ti{\cL}(\btt^k)\!-\!\nabla\ti{\cL}(\btt^{k-1})\|_\infty\!+\!\phi\|\btt^k\!-\!\btt^{k-1}\|_\infty \\
    &\!\leq\! (\kappa_++\phi)\|\btt^k\!-\!\btt^{k-1}\|_2\leq (1+\gamma_u)\rho_+\|\btt^k\!-\!\btt^{k-1}\|_2.
    \end{split}
    \label{eq:a4}
    \end{equation}
     In \eqref{eq:a4}$, \|\nabla\ti{\cL}(\btt^k)\!-\!\nabla\ti{\cL}(\btt^{k-1})\|_\infty\!$ is upper bounded by $\kappa_+\|\btt^k\!-\!\btt^{k-1}\|_2$ because the estimators in the second stage are within the localized cone near $\bbeta^*$ where LSE holds with high probability.

    The proof is completed by taking $\sup$ over $\|\ub\|_1\leq 1$ in the inequality above.
\end{proof}

\begin{proof}[Proof of Lemma \ref{lem:diff2}]

Recall the stopping criteria of the inflation of the quadratic isotropic parameter $\phi$, we have
    \#\label{0831.1}
    F(\btt^k)-F(\btt^{k-1})\leq \Psi_{\lambda,\phi}(\btt^k,\btt^{k-1})-F(\btt^{k-1}).
    \#
    The convexity of $\|\btt\|_1$ implies
    \begin{align*}
    \lambda\|\btt^{k-1}\|_1&\geq \lambda\|\btt^k\|_1+\big\langle \lambda\bxi^k,\btt^{k-1}-\btt^k\big\rangle.
    \end{align*}
    Therefore we obtain
    \begin{align}\label{e2}
    \!\!\!F(\btt^{k-1})&\!\geq\!\tilde{\cL}(\btt^{k-1})\!+ \lambda\|\btt^k\|_1+\big\langle \lambda\bxi^k,\btt^{k-1}-\btt^k\big\rangle.
    \end{align}
Given that
    \begin{equation}
    \begin{split}
    \Psi_{\lambda,\phi}(\btt^k, \btt^{k-1})=&\cL(\btt^{k-1})\!+\!\big\langle \nabla\cL(\btt^{k-1}),\btt^k\!-\!\btt^{k-1}\big\rangle+\frac{\phi}{2}\|\btt^k\!-\!\btt^{k-1}\|_2^2\!+p_{\lambda}(\bbeta^k) \\
    =&\tilde{\cL}(\btt^{k-1})\!+\!\big\langle \nabla\tilde{\cL}(\btt^{k-1}),\btt^k\!-\!\btt^{k-1}\big\rangle+\frac{\phi}{2}\|\btt^k\!-\!\btt^{k-1}\|_2^2\!+\lambda\norm{\bbeta^k}_1 \\
    &+h(\bbeta^k)-h(\bbeta^{k-1})-\!\big\langle h'(\btt^{k-1}),\btt^k\!-\!\btt^{k-1}\big\rangle \\
    \leq & \tilde{\cL}(\btt^{k-1})\!+\!\big\langle \nabla\tilde{\cL}(\btt^{k-1}),\btt^k\!-\!\btt^{k-1}\big\rangle+\frac{\phi}{2}\|\btt^k\!-\!\btt^{k-1}\|_2^2\!+\lambda\norm{\bbeta^k}_1
    \end{split}
    \label{e3}
    \end{equation}
    because of the concavity of function $h(\cdot)$, by
    plugging (\ref{e2}) and (\ref{e3}) back into (\ref{0831.1}), we obtain
    \begin{equation}\label{0107.7}
    F(\btt^k)\!-\!F(\btt^{k-1})\!\leq\!-\frac{\phi}{2}\|\btt^k\!-\!\btt^{k-1}\|_2^2+\!\langle\nabla\tilde{\cL}(\btt^{k-1})\!+\!\lambda\bxi^k,\btt\!-\!\btt^k\rangle.
    \end{equation}
By the first order optimality condition,  there exists some $\bxi^k$ such that
    \begin{equation}
    \nabla\tilde{\cL}(\btt^{k-1})+\phi (\btt^k-\btt^{k-1})+\lambda\bxi^k=0.
    \label{eq:optimality}
    \end{equation}
Plugging the optimality equation back to \eqref{0107.7},
we complete the proof.
\end{proof}


\begin{lemma}[Geometric Rate in the second stage]\label{lem:fg}
Under the same conditions for Theorem \ref{thm:tscomplexity}, $\{\btt^k\}$ converges geometrically,
\$
&F\big(\btt^k\big)\!-\!F\big(\hbt\big)\\
&\qquad\leq \Bigl(1\!-\!\frac{1}{4\gamma_u\kappa}\Bigr)^k\Big\{F(\btt^{2,0})\!-\!F(\hbt)\Big\}.
\$
for some constant $\kappa>0$ where $\hbt$ is the global minimizer defined in (\ref{eqloss}).
\end{lemma}

\begin{proof}[Proof of Lemma \ref{lem:fg}]

Define $\btt(\alpha)=\alpha \hbt +(1-\alpha)\btt^{k-1}$. Since $F(\btt^k)$ is majorized at  $\Psi(\btt^k, \btt^{k-1})$, we have
\$
\!F(\btt^k)&\!\leq\! \Psi(\btt^k, \btt^{k-1})\\
&\leq \min_{\btt(\alpha)} \Big\{  {\tilde{\cL}(\btt^{k-1})\!+\!\langle\nabla\tilde{\cL}(\bbeta^{k-1}),\btt\!-\!\btt^{k-1}\rangle} \!+\!\frac{\phi}{2}\|\btt\!-\!\btt^{k-1}\|_2^2\!+\lambda\|\btt\|_1\Big\}\notag\\
&\!\leq\! \min_{\btt(\alpha)}\Big\{F(\bbeta)+\frac{\phi}{2}\|\btt-\btt^{k-1}\|_2^2 \Big\},
\$
where we restrict $\btt$  on the line segment $\alpha \hbt+(1-\alpha)\btt^{k-1}$ in the first inequality and    the last inequality follows from the convexity of $\tilde{\cL}(\btt)$.
Using the convexity of $F(\btt)$, we obtain that
\begin{equation}
\begin{split}
F(\btt^k) &\!\leq\! \min_{\btt(\alpha)}\Big\{F(\btt)+\frac{\phi}{2}\|\btt-\btt^{k-1}\|_2^2\Big\}\\
    &\!\leq\! \min_{\alpha}\Big\{ \alpha F(\hbt)+(1-\alpha) F(\btt^{k-1})+\frac{\alpha^2\phi}{2}\|\btt^{k-1}-\hbt\|_2^2\Big\}\\
    &\!\leq\! \min_{\alpha}\Big\{F(\btt^{k-1})\!-\!\alpha \big[F(\btt^{k-1})\!-\!F(\hbt)\big]\!+\!\frac{\alpha^2\phi}{2}{\|\btt^{k-1}\!-\!\hbt\|_2^2}\Big\}.
\end{split}
\label{eq:a10}
\end{equation}
\noindent   Next,  we  bound the last term in the inequality above. Applying Lemma \ref{lem:3.15}, we obtain
\begin{align*}
    &\|(\btt^{k-1})_{S^c}\|_0\leq \ti s,~ \|\btt^{k-1}-\bttc\|_1\leq C'\lambda s\leq r,~ \|\hbt-\bttc\|_2\leq r, ~\mbox{and}~\|\hbt_{S^c}\|_0\leq \ti s.
\end{align*}

Recall $\hat\bxi$ is some subgradient of $\|\widehat\bbeta\|_1$. Using the convexity of $\tilde{\cL}(\cdot)$ and the $\ell_1$-norm,
 $F(\btt^{k-1})\!-\!F(\hbt)$ can be bounded in the following way
 \begin{equation}
 \begin{split}
F(\btt^{k-1})\!-\!F(\hbt)&\!\geq\! \big\langle\nabla\tilde{\cL}(\hbt)\!+\!\lambda\hat\bxi,\btt^{k-1}\!-\!\hbt\big\rangle \!+\left(\ti{\cL}(\bbeta^{k-1})-\ti{\cL}(\hbt)-\langle\nabla\ti{\cL}(\hbt),\bbeta^{k-1}-\hbt\rangle\right) \\
&\!\geq\! \frac{\kappa_-}{2}\|\btt^{k-1}\!-\!\hbt\|_2^2,
\end{split}
\label{eq:a11}
\end{equation}
where the last inequality is due to the first order optimality condition and the LSE condition.

Plugging \eqref{eq:a11} back to \eqref{eq:a10}, we conclude that
\$
F(\btt^k)&\!\leq\! \min\limits_{\alpha}\Big\{F(\btt^{k-1})\!-\!\alpha\big[F(\btt^{k-1})\!-\!F(\hbt)\big]\\
&\qquad\qquad\!+\!\frac{\alpha^2\phi}{\kappa_-}\big[F(\btt^{k-1})\!-\!F(\hbt)\big]\Big\}\\
&\leq F(\btt^{k-1})-\frac{\kappa_-}{4\phi}\big[F(\btt^{k-1})-F(\hbt)\big].
\$
which, combining with the fact $\phi\leq \gamma_u\rho_+$, yields
\$
F\big(\btt^k\big)-F\big(\hbt\big)&\leq \Big(1-\frac{1}{4\gamma_u\kappa}\Big)^k\Big\{F(\tilde\btt^{(0)})-F(\hbt)\Big\},
\$
in which $\kappa={\rho_+}/{\kappa_-}$.

\end{proof}

The next lemma is related to the parameter estimation and objective function bound for sparse approximate solutions.

\begin{lemma}\label{0908.1}
Let $\lambda\!\geq\! 2\left(\|\nabla\cL(\btt^*)\|_\infty\!+\!\varepsilon\right)$.
	If $\|(\btt-\bttc)_{S^c}\|_0\leq \ti s$, $\omega_{\lambda}(\btt)\leq \varepsilon$ and $\btt\in B_1(r,\bttc)$, then we must have
	\begin{gather*}
	\|\btt-\bttc\|_2\leq 3\kappa_-^{-1}\lambda\sqrt{s}/2,\\
	F(\btt)-F(\bttc)\leq 3\varepsilon\kappa_-^{-1}\lambda s.
	\end{gather*}
\end{lemma}

\begin{proof}[Proof of Lemma \ref{0908.1}]

Following the same argument in the proof of Proposition \ref{prop:decomposition2}, we have
	\#\label{0908.1.8}
   \|\btt-\bttc\|_2&
   \leq \frac{3\lambda\sqrt{s}}{2\kappa_-}.
	\#

	Next, we prove the desired  bound for $F(\btt)-F(\bttc)$. Using the convexity of $F(\cdot)$, we obtain
	\begin{align*}
	F(\bttc)\geq F(\btt)+\big\langle\nabla\cL(\btt)+p'_{\lambda}(\bbeta),\bttc-\btt\big\rangle,
	\end{align*}
which yields that
\#\label{0908.1.9}
F(\btt)-F(\bttc)\leq -\big\langle\nabla\cL(\btt)+p'_{\lambda}(\bbeta),\bttc-\btt\big\rangle\leq \varepsilon \|\bttc-\btt\|_1.
\#
	On the other hand, we know from  Lemma \ref{lem:akkt1} that the approximate solution $\btt$ falls in the $\ell_1$ cone:
	\$
	\|(\btt-\bttc)_{S^c}\|_1&\leq \|(\btt-\bttc)_{S}\|_1,
	\$
which, together with \eqref{0908.1.8}, implies
	\#\label{0908.1.10}
	\|\btt-\bttc\|_1&\leq 2\|(\btt-\bttc)_{S}\|_1 \leq 2\sqrt{s}\|(\btt-\bttc)_{S}\|_2 \leq 3\kappa_-^{-1} \lambda s.
	\#
	
	\noindent	Plugging  \eqref{0908.1.10} into \eqref{0908.1.9} completes the proof.
\end{proof}

Lemma \ref{lem:basicineq}, Lemma \ref{lem:0822.1}, Lemma \ref{lem:stepsparsity}  and Lemma \ref{lem:3.15} are general results for mid-stage estimators.
\begin{lemma}[Basic Inequality]\label{lem:basicineq}
Let $\lambda\!\geq\! 2\left(\|\nabla\cL(\btt^*)\|_\infty\!+\!\varepsilon\right)$. If $\|\btt_{S^c}\|_0\leq \ti s$, $\btt\in B_2(r,\bttc)$ and $F(\btt)-F(\bttc)\leq C\lambda^2s$, then
	\$
\frac{\kappa_-}{2}\|\btt-\bttc\|_2^2+\frac{\lambda}{2}\|(\btt-\bttc)_{S^c}\|_1\leq \frac{5\lambda}{2}\|(\btt-\bttc)_{S}\|_1+C\lambda^2s.
	\$
\end{lemma}

\begin{proof}[Proof of Lemma \ref{lem:basicineq}]
Since $\|\btt_{S^c}\|_0\leq \ti s$ and $\|\bttc_{S^c}\|_0=0$, we have $\|(\btt-\bttc)_{S^c}\|_0\leq \ti s$. The localized sparse strong convexity implies that
	\#\label{0822.1.1}
	\tilde{\cL}(\bttc)+\big\langle\nabla\tilde{\cL}(\bttc),\btt-\bttc\big\rangle+\frac{\kappa_-}{2}\|\btt-\bttc\|_2^2\leq\tilde{\cL}(\btt).
	\#
	Recall that $F(\btt)=\tilde{\cL}(\btt)+\lambda\|\btt\|_1$,
	$
	F(\btt)-F(\bttc)\leq C\lambda^2s,
	$
is equivalent to
	\#\label{0822.1.2}
	\tilde{\cL}(\btt)-\tilde{\cL}(\bttc)+\lambda(\|\btt\|_1-\|\bttc\|_1)\leq C\lambda^2s.
	\#
Plugging   \eqref{0822.1.1} into the left-hand side of \eqref{0822.1.2}, we   immediately obtain
	\begin{align*}
	\frac{\kappa_-}{2}\|\btt\!-\!\bttc\|_2^2 &\leq\! C\lambda^2s -\big\langle\nabla\tilde{\cL}(\bttc),\btt\!-\!\bttc\big\rangle\!+\!\lambda(\|\bttc\|_1\!-\!\|\btt\|_1)\\
	&=\! C\lambda^2s\underbrace{-\big\langle\nabla\cL(\bttc),\btt\!-\!\bttc\big\rangle}_{\Rom{1}}\!+\underbrace{\langle\lambda\bxi^*,\bbeta-\bbeta^*\rangle}_{\Rom{2}}+\!\underbrace{\lambda\left(\|\bttc\|_1\!-\!\|\btt\|_1\right)}_{\Rom{3}}, \\
	\end{align*}
where $\bxi^*\in\partial\norm{\bbeta^*}_1$.
Following a similar argument in the proof of Proposition \ref{prop:decomposition2}, we have
\begin{gather*}
\Rom{1}\leq \|(\btt-\bttc)_{S^c}\|_1\|\nabla\cL(\bttc)\|_\infty+\|(\btt-\bttc)_{S}\|_1\|\nabla\cL(\bttc)\|_\infty,\\
\Rom{2} \leq \lambda\norm{(\bbeta-\bbeta^*)_{S}}_1 \\
\Rom{3}\leq \lambda\|(\btt-\bttc)_{S}\|_1-\lambda\|(\btt-\bttc)_{S^c}\|_1.
\end{gather*}

Therefore, we have
	\$
	&\frac{\kappa_-}{2}\|\btt-\bttc\|_2^2+(\lambda-\|\nabla\cL(\bttc)\|_\infty)\|(\btt-\bttc)_{S^c}\|_1\\
	&\qquad{}\leq (2\lambda+\|\nabla\cL(\bttc)\|_\infty)\|(\btt-\bttc)_{S}\|_1+C\lambda^2s.
	\$
The proof is finished by noticing   that $\|\nabla\cL(\bbeta^*)\|_\infty\leq \lambda/2$.
\end{proof}

\begin{lemma}\label{lem:0822.1}  Let  $\|\nabla\cL(\bttc)\|_\infty+\varepsilon\leq \lambda/2$. 	If  $\btt\in B_1(r,\bttc)$ satisfies
	$
	\|\btt_{S^c}\|_0\leq \ti s$ and $F(\btt)-F(\bttc)\leq C\lambda^2 s,
	$
	then we must have
	\begin{gather}
	\|\btt-\bttc\|_2\leq C'\lambda\sqrt{s}, \notag\\
	\big\langle\nabla\ti{\cL}(\btt)-\nabla\ti{\cL}(\bttc),\btt-\bttc\big\rangle\leq C'^2\kappa_+(s+\ti s,r)\lambda^2 s, \notag
	\end{gather}
	where $C'=\max\{2\sqrt{C/\kappa_-}, 10/ \kappa_-\}$.
\end{lemma}
\begin{proof}[Proof of Lemma \ref{lem:0822.1}]
	Directly applying Lemma \ref{lem:basicineq}, it follows that
	\$
	\frac{\kappa_-}{2}\|\btt-\bttc\|_2^2&\leq \frac{5\lambda}{2}\|(\btt-\bttc)_{S}\|_1+C\lambda^2s.
	\$

To further bound the right-hand side of the inequality above, we discuss two cases regarding the magnitude of $\|(\btt-\bttc)_{S}\|_1$ with respect to $\lambda s$:
\begin{itemize}
	\item If $5\lambda\|(\btt-\bttc)_{S}\|_1/2\leq C\lambda^2s$, we have
\#\label{0822.1.8}
\frac{\kappa_-}{2}\|\btt-\bttc\|_2^2\leq 2C\lambda^2s, ~\textnormal{and thus}~\|\btt-\bttc\|_2\leq 2\sqrt{\frac{C}{\kappa_-}}\lambda\sqrt{s}.
\#
\item If $5\lambda\|(\btt-\bttc)_{S}\|_1/2> C\lambda^2s$, we have
\$
\frac{\kappa_-}{2}\|\btt-\bttc\|_2^2\leq 5\lambda\|(\btt-\bttc)_S\|_1\leq 5\lambda\sqrt{s}\|\btt-\btt^*\|_2,
\$
which further yields
\#\label{0822.1.9}
\|\bbeta-\bbeta^*\|_2\leq \frac{10}{\kappa_-}\lambda\sqrt{s}.
\#
	\end{itemize}
Combining (\ref{0822.1.8}) and (\ref{0822.1.9}), we obtain	
\$
\|\btt-\bttc\|^2_2&\leq \max\bigg\{2\sqrt{\frac{C}{\kappa_-}}, \frac{10}{\kappa_-}\bigg\}\lambda\sqrt{s}=C'\lambda\sqrt{s}, \$
where $C'=\max\{2\sqrt{C/\kappa_-}, 10/ \kappa_-\}$.

Naturally, we obtain
	\$\label{}
	\big\langle \ti{\cL}(\btt)-\ti{\cL}(\bttc),\btt-\bttc\big\rangle
	\leq \kappa_+ \|\btt-\bttc\|_2^2
	\leq C'^2\kappa_+\lambda^2s.
	\$
	
	\noindent This completes the proof.
\end{proof}

\begin{lemma}\label{lem:3.15}
Assume the same conditions in Theorem \ref{thm:tscomplexity} hold.
The solution sequence $\{\btt^k\}_{k=0}^\infty$ always satisfies that
	\begin{gather*}
	 F(\btt^{k})-F(\bttc)\leq C\lambda^2s,  \\
	 \|\btt^{k}_{S^c}\|_0\leq \ti{s},~\mbox{and}~\|\btt^{k}-\bttc\|_2\leq C'\lambda\sqrt{s}.
	\end{gather*}
	for $k\geq 0$, where  $C=3/(2\kappa_-)$ and $C'= 10/\kappa_-$.
\end{lemma}

\begin{proof}[Proof of Lemma \ref{lem:3.15}]
We prove the theorem by mathematical induction on $k$.\\
	
\noindent	{\bf Base case}: The stopping criterion in the first stage implies that  $\omega^1_{\lambda}(\btt^{0})\leq \varepsilon$. On the other hand, the optimality condition in the second stage can be written as
\$
	\omega^2_{\lambda}(\btt^0)&=\min_{\bxi\in \partial p_{\lambda}(\bbeta^0)} \Big\{\|\nabla\cL(\btt^{0})+\bxi\|_\infty\Big\}
\$
which,  together with the triangle inequality, yields
\$
\omega^2_{\lambda}(\btt^0)
	&\!\leq\! \min_{\bxi_1\in\partial\norm{\bbeta^0}_1,\bxi_2\in\partial p_{\lambda}(\bbeta^0)} \Big\{\|\nabla\cL(\btt^{0})\!+\!\lambda\bxi_1\|_\infty\!+\!\|\lambda\bxi_1-\bxi_2\|_\infty    \Big \}.
\$
Given that $\varepsilon\leq \lambda/2$ and $\!\|\lambda\bxi_1-\bxi_2\|_\infty\leq\lambda$, we obtain
\$
\omega^2_{\lambda}(\btt^{0})\leq 3\lambda/2,
\$
Thus the initialization satisfies that
	\$
	\|(\btt^0)_{\cE_\ell^c}\|_0\leq \ti s,~\omega_{\lambda}(\btt^{0})\leq 3\lambda/2,~\mbox{and}~ \phi\leq \gamma_u\rho_+(2s+2\ti s,r).
	\$
Therefore, using  Lemma \ref{0908.1}, we obtain
	\$
	F(\btt^0)-F(\bttc)\leq C\lambda^2s, ~~\mbox{where}~~C=3/(2\kappa_-).
	\$
	Therefore, directly applying Lemma \ref{lem:0822.1} results
		\$
	\|\btt^0-\bttc\|_2\leq C'\lambda\sqrt{s},
	\$
	where $C'=10/\kappa_-$. \\

\noindent {\bf Induction step}:
\noindent	Suppose that, at the $(k\!-\!1)$-th iteration of the LAMM method in the second stage, we have
	\$
	\|(\btt^{k-1})_{S^c}\|_0\leq \ti s,~\phi\leq \gamma_u\rho_+, ~\mbox{and}~F(\btt^{k-1})-F(\bttc)\leq C\lambda^2s.
	\$
 Then according to Lemma \ref{lem:stepsparsity}, we have that the solution to the  LAMM method at the $k$th iteration is $(s+\ti s)$-sparse:   $\btt^{k}=T_{\phi,\lambda}(\btt^{k-1})$ satisfies
	$
	\|(\btt^{k})_{S^c}\|_0\leq \ti s.
	$
	Thus Lemma \ref{lem:diff2} implies	\$
	F(\btt^{k})\leq F(\btt^{k-1})- \frac{\phi}{2}\|\btt^{k}-\btt^{k-1}\|.
	\$
	which implies that
	\$
	F(\btt^{k})-F(\bttc)\leq F(\btt^{k-1})-F(\bttc)-\frac{\phi}{2}\|\btt^{k}-\btt^{k-1}\|_2^2\leq C\lambda^2s.
	\$
Therefore we have the induction holds at the $k$th iteration:
	\$
	\|(\btt^{k})_{S^{c}}\|_0\leq \ti s,\phi\leq \gamma_u\rho_+(2s+2\ti s), ~\mbox{and}~F(\btt^{k})-F(\bttc)\leq C\lambda^2s.
	\$
Using Lemma \ref{lem:0822.1},  for $C'$ defined as before, we obtain
	\$
	\|\btt^{k}-\bttc\|_2\leq C'\lambda\sqrt{s}.
	\$
We complete induction on $k$.
\end{proof}

\begin{lemma}\label{lem:stepsparsity}
   Let  $\|\nabla\cL(\bttc)\|_\infty+\varepsilon\leq \lambda/2$.  	Let $\btt\in B_1(r,\bttc)$ satisfy
	$
	\|\btt_{S^c}\|_0\leq \ti s$ and $F(\btt)-F(\bttc)\leq C\lambda^2s$. Let $C_0=212\gamma_u\rho_+/\kappa_-+1600\left(\rho_+/\kappa_-\right)^2$. If $\ti s\geq C_0s$,
 then the one-step LAMM algorithm produces a $(s+\ti s)$-sparse solution: 
$\label{}
\|(T_{\ti{\cL},\lambda,\phi}(\btt))_{S^c}\|_0\leq \ti s.
$
\end{lemma}

\begin{proof}[Proof of Lemma \ref{lem:stepsparsity}]
	For simplicity, we write $\bar{\btt}\!=\!\btt-{\phi}^{-1}{\nabla\ti{\cL}(\btt)}$. To show that  $\|\big(S(\bar\btt,{\phi}^{-1}{\blam})\big)_{S^c}\|_0\!\leq\! \ti s$, it suffices to prove that, for any  $j\!\in\! \cE^c$, the  total number of $\beta_j$'s such that $\bar{\beta}_j\!>\!\lambda_j/\phi$ is no more than $\ti s$. 
We first write $\bar\btt$ as
	\$
	\bar\btt&=\btt-\frac{1}{\phi}\nabla\ti{\cL}(\btt)=\btt-\frac{1}{\phi}\nabla\ti{\cL}(\bttc)+\frac{1}{\phi}\nabla\ti{\cL}(\bttc)-\frac{1}{\phi}\nabla\ti{\cL}(\btt).
	\$
	Define   $\ti{S}=\{j\in S^c: (\btt-{\phi}^{-1}{\nabla\ti{\cL}(\btt)})_j=\lambda_j/\phi\}$, and notice that $\{j:(T_{\blam,\phi}(\btt))_j\neq 0\}\subseteq \ti{S}$, thus it suffices to show $\vert\ti{S}\vert\leq \ti s$.
We further  define $S^1$, $S^2$ and $S^3$ as:
	\begin{align}\label{}
	S^1&\equiv \Big\{j\in S^c: |\btt_j|\geq \frac{1}{4}\cdot \frac{\lambda}{\phi}\Big\}, \label{0819.1}\\
	S^2&\equiv \Big\{j\in S^c: |\nabla\ti{\cL}(\bttc)_j/\phi|>\frac{1}{2}\cdot\frac{\lambda}{\phi}\Big\},  \label{0819.2}\\
	S^3&\equiv \Big\{j\in S^c: \Big|\Big(\frac{\nabla\ti{\cL}(\btt)-\nabla\ti{\cL}(\bttc)}{\phi}\Big)_j\Big|>\frac{1}{4}\cdot\frac{\lambda}{\phi}\Big\}.  \label{0819.3}
	\end{align}
 We immediately have $\ti S\subseteq S^1\cup S^2\cup S^3$. It suffices to prove that $|S^1|+|S^2|+|S^3|\leq \ti s$. The assumption that $\|\nabla\cL(\bttc)\|_\infty\!+\!\varepsilon\!\leq\! \lambda/2$ implies $S^2=\emptyset$.
In what follows, we bound $|S^1|$ and $|S^3|$,  respectively. \\

\noindent \textbf{{Bound for $|S^1|$:}}\\
\noindent For $\forall j\in S^c$, we have $\beta_j^\circ=0$. Using Markov inequality, we obtain
	\$
	|S^1|&=\Big|\Big\{j\in S^c: |\btt_j|\geq \frac{1}{4}\cdot \frac{\lambda}{\phi}\Big\}\Big|	\leq \sum_{j\in S^c} \frac{4\phi}{\lambda}|\beta_j-\beta_j^*|\leq\frac{4\phi}{\lambda}\|(\btt-\bttc)_{S^c}\|_1.
	\$

It remains to bound  $\|(\btt-\bttc)_{S^c}\|_1$. 
According to Lemma \ref{lem:basicineq} and Lemma \ref{lem:3.15},
	\$\label{}
	\frac{\lambda}{2}\|(\btt-\btt^*)_{S^c}\|_1\leq \frac{5\lambda}{2}\|(\btt-\bttc)_{S}\|_1+3/2\kappa_-^{-1}\lambda^2s.
	\$

Therefore,  $\btt-\bttc$ falls in the approximate $\ell_1$ cone:
	\$
	\|(\btt-\bttc)_{S^c}\|_1
	&\leq5\|(\btt-\bttc)_{S}\|_1 +3\kappa_-^{-1}\lambda s\leq 5C'\lambda s+3\kappa_-^{-1}\lambda s=53\kappa_-^{-1}\lambda s.
	\$
Thus
	\$\label{}
	|S^1|\leq 212\phi\kappa_-^{-1} s\leq {212\kappa_-^{-1}\gamma_u\rho_+}s,
	\$
	where we use the fact $\phi\leq \gamma_u\rho_+$ in the last inequality.\\
	
	
\noindent	\textbf{{Bound for $|S^3|$:}}\\
 Consider an arbitrary subset $S'\subseteq S^3$ with  size $|S'|=s'\leq \ti{s}$.
	Let us further consider  a $p$-dimensional sign vector $\ub$ such that $\|\mbu\|_\infty=1$ and $\|\mbu\|_0=s'$. 
There exists some $\ub$ such that
	\$
	\frac{1}{4}\lambda s'&\leq \sum_{j\in S^c} \frac{1}{4}\lambda_j |u_j|\leq \mbu^T \big\{\nabla\ti{\cL}(\btt)-\nabla\ti{\cL}(\bttc)\big\}.
	\$
 By the Mean Value theorem,  there exists some $\gamma\in[0,1]$ such that $\nabla\ti{\cL}(\btt)-\nabla\ti{\cL}(\bttc)=\big[\nabla^2\ti{\cL}\big(\gamma\btt+(1-\gamma)\bttc\big)\big](\btt-\bttc)$.  
 Let  $\Hb\equiv\big[\nabla^2\ti{\cL}(\gamma\btt+(1-\gamma)\bttc)\big]$. Writing  $\ub^T(\nabla\ti{\cL}(\btt)-\nabla\ti{\cL}(\bttc))$ as $\langle\Hb^{1/2}\ub,\Hb^{1/2}(\btt-\bttc) \rangle$ and applying the H{\"o}lder inequality, we obtain
\#\label{0108.6}
\lambda s'	/4&\!\leq\! \|\Hb^{1/2}\ub\|_2\|\Hb^{1/2}(\btt\!-\!\bttc)\|_2
	\!\leq\! \sqrt{\kappa_+(s',r) s'}\underbrace{\|\Hb^{1/2}(\btt\!-\!\bttc)\|_2}_{\Rom{1}}.
\#
To bound term \Rom{1},  we apply Lemma \ref{lem:3.15} and obtain that
\$
\Rom{1}= \|\Hb^{1/2}(\btt-\bttc)\|_2\leq 10\kappa_-^{-1}\sqrt{\kappa_+(2s+\ti s, r)}\lambda\sqrt{s}.
 \$
Plugging  the above inequality  into \eqref{0108.6},  we obtain
\$
\lambda s'/4\leq \sqrt{\rho_+(s',r)}\sqrt{s'}\times  10\kappa_-^{-1}\sqrt{\rho_+(2s+2\ti s,r)}\lambda\sqrt s.
\$
Taking squares of both sides yields
\$
s'\leq {1600\kappa_-^{-2}\rho_+(s',r)\rho_+}s\leq {1600\kappa_-^{-2}\rho_+^2}s<\ti s,
\$
where the last inequality is due to Condition 1 with $C_1=1600$. Since $s'=|S'|$ achieves the maximum possible value such that $s'\leq \ti s$ for any subset $S'$ of $S^3$ and the above inequality shows that $s'<\ti s$, we must have $S'= S^3$ and 
\$
\abs{S^3}\leq 1600\kappa_-^{-2}\rho_+^2 s<\ti s.
\$
Finally, combining  bounds for $|S^1|$, $|S^2|$ and $|S^3|$, we obtain
	\$
	\|(T_{\ti{\cL},\blam,\phi}(\btt))_{S^c}\|_0&\leq {212\gamma_u\kappa_-^{-1}\rho_+}s+{1600\kappa_-^{-2}\rho_+^2}s\leq \ti s
	\$
due to Condition \ref{cond:lse}.
	
\end{proof}

\end{section}


\begin{section}{Details and Proof of Theorem \ref{thm:minse}}

In \cite{huang2013oracle}, they discussed the M-estimator in Cox's model with Lasso penalty and proved the restricted strong convexity of Cox's model's loss function within a cone near the true signal $\bbeta^*$. We borrow some of their notations and techniques in this section to reveal more details about Theorem \ref{thm:minse} and prove the theorem in a parallel way.

\begin{subsection}{The constant $C^*$ and $C_*$}

Assume that $\mathbb{P}\{\sup\limits_t\norm{X_i(t)}_{\infty}\leq M\}=1$ for some constant $M$. For simplicity, let \[\bS^{(k)}(t,\bbeta)=\frac{1}{n}\sum\limits_{i=1}^n\bX_i(t)^{\otimes k}Y_i(t)e^{\bbeta^{\top}\bX_i(t)},\quad k=0,1,2.\]
\[\bR_n(t,\bbeta)=\frac{1}{n}\sum\limits_{i=1}^n Y_i(t)e^{\bbeta^{\top}\bX_i(t)},\quad \bar{\bX}_n(t,\bbeta))=\frac{\bS^1(t,\bbeta)}{\bS^{(0)}(t,\bbeta)}.\]
\[\bV_n(t,\bbeta)=\frac{1}{n}\sum\limits_{i=1}^n\frac{Y_i(t)e^{\bbeta^{\top}\bX_i(t)}}{\bS^{(0)}(t,\bbeta)}(\bX_i(t)-\bar{\bX}_n(t,\bbeta))^{\otimes 2} = \frac{\bS^2(t,\bbeta)}{\bS^{(0)}(t,\bbeta)}-\bar{\bX}_n(t,\bbeta)^{\otimes 2}.\]

With these notations, the gradient is written as \[\nabla\cL(\bbeta)=-\frac{1}{n}\sum\limits_{i=1}^n\int_0^{t^*}\left[\bX_i(s)-\bar{\bX}_n(s,\bbeta)\right]\mathrm{d}N_i(s),\] and the Hessian matrix of $\cL(\bbeta)$ is \[\nabla^2\cL(\bbeta)=\frac{1}{n}\int_0^{t^*}\bV_n(s,\bbeta)\mathrm{d}\bar{N}(s)=\frac{1}{n}\int_0^{t^*}\sum\limits_{i=1}^n\left\{\bX_i(s)-\bar{\bX}_n(s;\bbeta)\right\}^{\otimes 2}Y_i(s)\exp(\bbeta^{\top}\bX_i(s))\mathrm{d}\Lambda_0(s).\] for some positive $t^*$.

We write the population version of the Hessian matrix as
\begin{equation}
\bSigma(t^*,\bbeta)=\mathbb{E}\int_0^{t^*}\left\{\bX(s)-\bmu(s,\bbeta)\right\}^{\otimes 2}Y(s)\exp\left(\bbeta^{\top}\bX(s)\right)\mathrm{d}\Lambda_0(s)
\label{eq:truncatedhessian}
\end{equation}
with
\[\bmu(t,\bbeta)=\frac{\mathbb{E}\bX(t)Y(t)\exp(\bbeta^{\top}\bX(t))}{\mathbb{E}Y(t)\exp(\bbeta^{\top}\bX(t))}.\]

Let us define the minimum $s'$-sparse eigenvalue of a matrix.

\begin{dfn}
For any $s'\in\mathbb{Z}^+$,
\begin{enumerate}
\item Define $\pi_-(\bSigma,s')=\inf\limits_{\norm{\bb}_0\leq s'}\frac{(\bb^{\top}\bSigma\bb)^{1/2}}{\norm{\bb}_2}$ as the minimum $s'$-sparse eigenvalue of $\bSigma$.
\item Define $\pi_+(\bSigma,s')=\sup\limits_{\norm{\bb}_0\leq s'}\frac{(\bb^{\top}\bSigma\bb)^{1/2}}{\norm{\bb}_2}$ as the maximum $s'$-sparse eigenvalue of $\bSigma$.
\end{enumerate}
\label{def:minse}
\end{dfn}

The minimum and maximum $s'$-sparse eigenvalues are closely related to LSE in that $\pi_-(\nabla^2\cL(\bbeta^*),s')=\rho_-(s',0)$ and $\pi_+(\nabla^2\cL(\bbeta^*),s')=\rho_+(s',0)$. In Theorem \ref{thm:minse},
\begin{itemize}
\item $C_-(s')=\pi_-(\bSigma(t^*,\bbeta^*),s').$
\item $C_*$ is the smallest eigenvalue of $\bSigma(t^*,\bbeta^*)$.
\item $C_+(s')=\pi_+(\bSigma(t^*,\bbeta^*),s').$
\item $C^*$ is the largest eigenvalue of $\bSigma(t^*,\bbeta^*)$.
\end{itemize}

\end{subsection}

\begin{subsection}{Proof of Theorem \ref{thm:minse}}
The proof closely follows the proof of Theorem 4.1 in \cite{huang2013oracle}. The procedures of proving the probabilistic upper bound and lower bound of LSE are symmetric, hence we only provide the proof for the lower bound.

\begin{proof}
Define \[\hat{\bG}_n(t):=n^{-1}\sum\limits_{i=1}^n\left\{\bX_i-\bar{\bX}_n(t,\bbeta^*)\right\}^{\otimes 2}Y_i(t)\exp(\bbeta^{*\top}\bX_i(t)),\] \[\bG_n(t):=n^{-1}\sum\limits_{i=1}^n\left\{\bX_i-\bmu(t,\bbeta^*)\right\}^{\otimes 2}Y_i(t)\exp(\bbeta^{*\top}\bX_i(t)).\] The definition of $\bar{\bX}_n(t,\bbeta)$ and $\bmu(t,\bbeta)$ could be found in Section \ref{sec:eigen}.

With the notation above, we write the Hessian as $\ddot{\cL}(\bbeta)=\int_0^{t^*}\hat{\bG}_n(s,\bbeta)\mathrm{d}\Lambda_0(s)$ and its population version as $\bSigma(t^*,\bbeta)=\mathbb{E}\int_0^{t^*}\bG_n(s,\bbeta)\mathrm{d}\Lambda_0(s)$.

By the definition of $\hat{\bG}_n(t,\bbeta)$ and $\bG_n(t)$, we have $\bG_n(t,\bbeta)=\hat{\bG}_n(t,\bbeta)+\left\{\bar{\bX}_n(t,\bbeta)-\bmu(t,\bbeta)\right\}^{\otimes2}$. Hence,
\begin{equation}
\ddot{\cL}(\bbeta)=\int_0^{t^*}\bG_n(s,\bbeta)-\left\{\bar{\bX}_n(s,\bbeta)-\bmu(s,\bbeta)\right\}^{\otimes2}\mathrm{d}\Lambda_0(s).
\label{eq:1}
\end{equation}
We first bound the second term on the right hand side of \ref{eq:1}. Define \[R_n(t,\bbeta):=n^{-1}\sum\limits_{i=1}^nY_i(t)\exp(\bX_i^{\top}\bbeta),\]\[\bDelta(t,\bbeta):=R_n(t,\bbeta)\left\{\bar{\bX}_n(t,\bbeta)-\bmu(t,\bbeta)\right\}=n^{-1}\sum\limits_{i=1}^nY_i(t)\exp(\bX_i^{\top}\bbeta)\left\{\bar{\bX}_n(t,\bbeta)-\bmu(t,\bbeta)\right\}.\]
Since $Y_i(t)$ is non-increasing in $t$,
\begin{equation}
0\leq\int_0^{t^*}\left\{\bar{\bX}_n(s,\bbeta)-\bmu(s,\bbeta)\right\}\mathrm{d}\Lambda_0(s)\leq\frac{\int_0^{t^*}\bDelta^{\otimes2}(t,\bbeta)\mathrm{d}\Lambda_0(s)}{R_n^2(t^*,\bbeta)}
\label{eq:2}
\end{equation}
Since $R_n(t^*,\bbeta)$ is the average of i.i.d. variables uniformly bounded by $M$ and $\mathbb{E}R_n(t^*,\bbeta)=r_*$, the Hoeffding inequality gives
\[
\mathbb{P}\left(R_n(t^*,\bbeta)<r_*/2\right)\leq\exp(-nr_*^2/8M^2).
\]
Since $\bDelta(t,\bbeta)$ is an average of i.i.d. mean zero vectors, \[\left(n^2\int_0^{t^*}\bDelta^{\otimes2}(t,\bbeta)\mathrm{d}\Lambda_0(s)\right)_{jk}\] is a degenerate $V-$statistics for each $(j,k)$. Moreover, since the summands of these V-statistics are all bounded by $4M^2\Lambda_0(t^*)$, Lemma \ref{lem:Vstat} yields
\[\max\limits_{1\leq j,k\leq p}\mathbb{P}\left\{\pm\left(\int_0^{t^*}\bDelta^{\otimes2}(t,\bbeta)\mathrm{d}\Lambda_0(s)\right)_{jk}>4M^2\Lambda_0(t^*)t^2\right\}\leq3\exp\left(\frac{-nt^2/2}{1+t/3}\right).\]
Thus, by (\ref{eq:1}), (\ref{eq:2}), the above two probabilistic bounds and Lemma \ref{lem:se},
\begin{equation}
\pi_-(\ddot{\cL}(\bbeta^*),s)\geq\pi_-\left(\int_0^{t^*}\bG_n(t,\bbeta^*)\mathrm{d}\Lambda_0(t),s\right)-4sM^2\Lambda_0(t^*)t^2_{n,p,\varepsilon}/(r_*/2)
\label{eq:3}
\end{equation}
with probability at least $1-\varepsilon-\exp(-nr_*^2/8M^2)$.

Finally, $\int_0^{t^*}\bG_n(s,\bbeta)\mathrm{d}\Lambda_0(s)$ is an average of i.i.d. matrices with mean $\bSigma(t^*,\bbeta)$. The summands of $\left(\int_0^{t^*}\bG_n(s,\bbeta)\mathrm{d}\Lambda_0(s)\right)_{jk}$ are uniformly bounded by $4M^2\Lambda_0(t^*)$, so that the Hoeffding inequality gives
\begin{equation}
\mathbb{P}\left\{\max\limits_{j,k}\abs{\left(\int_0^{t^*}\bG_n(s,\bbeta)\mathrm{d}\Lambda_0(s)-\bSigma(t^*,\bbeta)\right)_{jk}}>4M^2\Lambda_0(t^*)\right\}\leq p(p+1)\exp(-nt^2/2)
\label{eq:4}
\end{equation}

By (\ref{eq:3}), (\ref{eq:4}) with $t=L_n(p(p+1)/\varepsilon)$ and Lemma \ref{lem:se}, we have
\begin{equation}
\pi_-(\ddot{\cL}(\bbeta^*),s)\geq\pi_-(\bSigma(t^*,\bbeta^*),s)-4sM^2\left[(1+\Lambda_0(t^*))L_n(p(p+1)/\varepsilon)+2\Lambda_0(t^*)t^2_{n,p,\varepsilon}/r_*\right]
\label{eq:a7}
\end{equation}
with probability at least $1-2\varepsilon-\exp(-nr_*^2/8M^2)$.

(\ref{eq:a7}) gives a probabilistic lower bound for the minimum $s$-sparse eigenvalue at $\ddot{\cL}(\bbeta^*)$. Now we extend this lower bound to the neighborhood of $\bbeta^*$.
According to Lemma \ref{lem:neighborhood}, we have
\begin{equation}
\begin{split}
&\rho_-(s,r) \\
 \geq&\exp\left(-2\sup\limits_{\norm{\bbeta-\bbeta^*}_1\leq r, \norm{\bbeta}_0\leq s}\max\limits_{t\geq 0}\max\limits_{i,j}\abs{(\bbeta-\bbeta^*)^{\top}(X_i(t)-X_j(t))}\right)\pi_-(\ddot{\cL}(\bbeta^*),s) \\
 \geq&\exp(-4rM)\pi_-(\ddot{\cL}(\bbeta^*),s) \\
\end{split}
\label{eq:a8}
\end{equation}

Combining (\ref{eq:a8}) with (\ref{eq:a7}) completes the proof.
\end{proof}

\begin{lemma}
Let $X_i$ be a sequence of independent stochastic processes and $f_{i,j}$ be functions of $X_i$ and $X_j$ with $\abs{f_{i,j}}\leq 1$. Suppose $f_{i,j}$ are degenerate in the sense of $\mathbb{E}\left[f_{i,j}\vert X_i\right]=\mathbb{E}\left[f_{i,j}\vert X_j\right]=0$ for all $i\neq j$. Let $V_n=\sum\limits_{i=1}^n\sum\limits_{j=1}^nf_{i,j}$. Then \[\mathbb{P}\left\{\pm V_n>(nt)^2\right\}\leq\frac{2\varepsilon_n(t)(1+\varepsilon_n(t))}{(1+\varepsilon_n^2(t))^2}\leq 3\exp\left(-\frac{nt^2/2}{1+t/3}\right),\] where $\varepsilon_n(t)=\exp\left(-\frac{nt^2/2}{1+t/3}\right)$.
\label{lem:Vstat}
\end{lemma}

\begin{lemma}
Let $\bar{\bSigma}$ and $\bSigma$ be two positive semi-definite matrices with elements $\bar{\Sigma}_{jk}$ and $\Sigma_{jk}$.
\begin{enumerate}
\item $\pi_-(\bar{\bSigma},s')\geq\pi_-(\bSigma,s')-s'\cdot\max\limits_{1\leq j\leq k\leq p}\abs{\bar{\Sigma}_{jk}-\Sigma_{jk}}$.
\item If $\bar{\bSigma}\succeq\bSigma$, then $\pi_-(\bar{\bSigma},s')\geq\pi_-(\bSigma,s')$.
\end{enumerate}
\label{lem:se}
\end{lemma}

\proof[Proof of Lemma \ref{lem:se}]
     \textcolor{white}{  }

\begin{enumerate}
\item For $\bu$ satisfies $\norm{\bu}_2=1$ and $\norm{\bu}_0\leq s'$, according to Cauchy-Schwarz inequality, \[\abs{\bu^{\top}\bar{\bSigma}\bu-\bu^{\top}\bSigma\bu}\leq\norm{\bu}_1^2\cdot\max\limits_{j,k}\abs{\bar{\Sigma}_{jk}-\Sigma_{jk}}\leq s'\norm{\bu}_2^2\cdot\max\limits_{j,k}\abs{\bar{\Sigma}_{jk}-\Sigma_{jk}}=s'\cdot\max\limits_{j,k}\abs{\bar{\Sigma}_{jk}-\Sigma_{jk}}.\]
\item The proof follows directly from Definition \ref{def:minse}.
\end{enumerate}

\begin{lemma} For any $\bbeta,\bb\in\mathbb{R}^p$, denote $\eta_{\bb}=\max\limits_{t\geq 0}\max\limits_{i,j}\abs{\bb^{\top}[\bX_i(t)-\bX_j(t)]}$, then
\[e^{-2\eta_{\bb}}\cdot\nabla^2\cL(\bbeta)\preceq\nabla^2\cL(\bbeta+\bb)\preceq e^{2\eta_{\bb}}\cdot\nabla^2\cL(\bbeta).\]
\label{lem:neighborhood}
\end{lemma}

Proof of Lemma \ref{lem:Vstat} and Lemma \ref{lem:neighborhood} are omitted here since they could be found in \cite{huang2013oracle}.

\end{subsection}
\end{section}

\begin{section}{Proof of Proposition \ref{prop:contraction}}

\begin{proof}
The proof for the $\ell_2$ bound of the first-stage estimator $\ti{\bbeta}^1$ can be found in Lemma 5.1 in \cite{fan2018lamm}.

Based on the $\ell_2$ bound and Lemma \ref{lem:akkt1}, we have
\[	\|(\tbt-\bttc)_{S^c}\|_1\leq \frac{\lambda/2}{\lambda-\lambda/2}\|{(\tbt-\bttc)}_{S}\|_1=\|{(\tbt-\bttc)}_{S}\|_1.\]
Therefore,
\begin{align*}
\norm{\ti{\bbeta}^1-\bbeta^*}_1=&\norm{(\ti{\bbeta}^1-\bbeta^*)_S}_1+\norm{(\ti{\bbeta}^1-\bbeta^*)_{S^c}}_1\leq2\norm{(\ti{\bbeta}^1-\bbeta^*)_S}_1\\
\leq & 2\sqrt{s}\norm{(\ti{\bbeta}^1-\bbeta^*)_S}_2\leq 2\sqrt{s}\norm{\ti{\bbeta}^1-\bbeta^*}_2\leq 36\rho_*\lambda s
\end{align*}
\end{proof}
\end{section}

\begin{section}{Proof of Proposition \ref{prop:maxgrad}}
\proof
The first probabilistic bound for $\norm{\nabla\cL(\bbeta^*)}_{\infty}$ is proved in Theorem 3.2 in \cite{huang2013oracle}.

Write $\ba_i(s)=\bX_i(s)-\bar{\bX}_n(s,\bbeta^*)$ and $\bA_i=\int_0^{t^*}\ba_{iS} \mathrm{d} N_i(s)$. We have
\begin{equation}
\begin{split}
\norm{\nabla\cL(\bbeta^*)_S}_2=&\sqrt{\frac{1}{n^2}\left(\sum\limits_{i=1}^n\int_0^{t^*}\left[\bX_i(s)-\bar{\bX}_n(s,\bbeta^*)\right]_S \mathrm{d} N_i(s)\right)^\top\left(\sum\limits_{i=1}^n\int_0^{t^*}\left[\bX_i(s)-\bar{\bX}_n(s,\bbeta^*)\right]_S \mathrm{d} N_i(s)\right)} \\
=&\sqrt{\frac{1}{n^2}\left(\sum\limits_{i=1}^n\int_0^{t^*}\ba_{iS} \mathrm{d} N_i(s)\right)^\top\left(\sum\limits_{i=1}^n\int_0^{t^*}\ba_{iS} \mathrm{d} N_i(s)\right)}=\sqrt{\frac{1}{n^2}\left(\sum\limits_{i=1}^n\bA_i\right)^\top\left(\sum\limits_{i=1}^n\bA_i\right)}
\end{split}
\label{eq:D3}
\end{equation}

According to Proof of Lemma 3.3 in \cite{huang2013oracle}, let $t_j$ be the time of the $j$th jump of the process $\sum\limits_{i=1}^n\int_0^{\infty} Y_i(t)\mathrm{d} N_i(t)$ and $t_0=0$. Then, for $j\geq0$, \[\bZ_j=\sum\limits_{i=1}^n\int_0^{t_j}\ba_{i}(s)\mathrm{d}N_i(s)\] is a martingale sequence with difference $\norm{\bZ_j-\bZ_{j-1}}_\infty\leq M$. Thus, $\mathbb{E}A_i=0$ and $A_i\indep A_j$ when $i\neq j$. Given that $\norm{\bA_i}_\infty\leq 2M$, we have

\begin{align}
\mathbb{E}\left[\frac{1}{n}\left(\sum\limits_{i=1}^n\bA_i\right)^\top\left(\sum\limits_{i=1}^n\bA_i\right)\right]= \mathbb{E}\left[\bA_1^\top\bA_1\right]\leq 4M^2s
\label{eq:D1}
\end{align}

It left for us to bound
\[
\frac{1}{n}\left(\sum\limits_{i=1}^n\bA_i\right)^\top\left(\sum\limits_{i=1}^n\bA_i\right)-\mathbb{E}\left[\frac{1}{n}\left(\sum\limits_{i=1}^n\bA_i\right)^\top\left(\sum\limits_{i=1}^n\bA_i\right)\right].
\]

Applying Theorem 3.4 in \cite{wainwright2019high} yields


\[
\mathbb{P}\left(\frac{1}{n}\left(\sum\limits_{i=1}^n\bA_i\right)^T\left(\sum\limits_{i=1}^n\bA_i\right)\geq\mathbb{E}\left[\frac{1}{n}\left(\sum\limits_{i=1}^n\bA\right)^T\left(\sum\limits_{i=1}^n\bA_i\right)\right]+t\right) \leq \exp\left(\frac{-t^2}{(32M^2s)^2}\right)
\]

Let $t=\gamma\mathbb{E}\left[\frac{1}{n}\left(\sum\limits_{i=1}^n\bA\right)^T\left(\sum\limits_{i=1}^n\bA_i\right)\right]$, then

\begin{align}
\mathbb{P}\left(\frac{1}{n}\left(\sum\limits_{i=1}^n\bA_i\right)^T\left(\sum\limits_{i=1}^n\bA_i\right)\geq(\gamma+1)\mathbb{E}\left[\frac{1}{n}\left(\sum\limits_{i=1}^n\bA\right)^T\left(\sum\limits_{i=1}^n\bA_i\right)\right]\right) \leq \exp\left(\frac{-\gamma^2}{64}\right)
\label{eq:D2}
\end{align}

Combining \eqref{eq:D1} and \eqref{eq:D2} and taking $\varepsilon_0=\exp(-\gamma^2/64)$ yields

\[
\mathbb{P}\left(\frac{1}{n}\left(\sum\limits_{i=1}^n\bA_i\right)^T\left(\sum\limits_{i=1}^n\bA_i\right)\geq4(\sqrt{-64\log \varepsilon_0}+1)M^2s\right) \leq \varepsilon_0.\]

This completes the proof together with \eqref{eq:D3}.

\end{section}

\begin{section}{{Proof of Proposition \ref{prop:decomposition2}}}

\begin{proof}
We write the penalized loss function as \[
\tilde{\cL}(\bbeta)+\lambda\norm{\bbeta}_1.\]

For ease of notation, we denote $\hat{\bbeta}=\ti{\bbeta}^2$ as the second-stage estimator in the proof of Proposition \ref{prop:decomposition2}.

We construct $\hat{\bbeta}^*=\bbeta^*+t(\hbbeta-\bbeta^*)$ and let $t$ be the largest $t\in(0,1)$ such that $\norm{\hat{\bbeta^*}-\bbeta^*}_1\leq r$. The way we construct $\hbbeta^*$ suggests that $t=1$ if $\norm{\hat{\bbeta}-\bbeta^*}_1\leq r$ and $t\in(0,1)$ otherwise. The construction ensures that $\hat{\bbeta}^*$ is within the strong convexity cone so that
\begin{equation}
\kappa_-\norm{\hat{\bbeta}^*-\bbeta^*}_2^2\leq \langle\nabla\tilde{\cL}(\hat{\bbeta}^*)-\nabla\tilde{\cL}(\bbeta^*),\hat{\bbeta}^*-\bbeta^*\rangle=:\cD_{\tilde{\cL}}^s(\hat{\bbeta}^*-\bbeta^*)\leq t\cD_{\tilde{\cL}}^s(\hat{\bbeta}-\bbeta^*).
\label{eq:cc}
\end{equation}

The first inequality in (\ref{eq:cc}) holds because $\ti{\bbeta}^1$ is $(\tilde{s}+2s)$-sparse and the sparsity level remains in the second stage according to Lemma \ref{lem:stepsparsity}.
The last inequality is proved in Lemma F.2 in \cite{fan2018lamm}. To bound $\cD_{\tilde{\cL}}^s(\hat{\bbeta}-\bbeta^*)$ from above,  let $h(\bbeta)=p_{\lambda}(\bbeta)-\lambda\norm{\bbeta}_1$.
\begin{align*}
\cD_{\tilde{\cL}}^s(\hbbeta,\bbeta^*)&=\langle\nabla\tilde{\cL}(\hat{\bbeta})+\lambda\hat{\bxi},\hat{\bbeta}-\bbeta^*\rangle-\langle\lambda\hbxi,\hat{\bbeta}-\bbeta^*\rangle-\langle\nabla\tilde{\cL}(\bbeta^*),\hat{\bbeta}-\bbeta^*\rangle \\
&=\langle\nabla\tilde{\cL}(\hat{\bbeta})+\lambda\hat{\bxi},\hat{\bbeta}-\bbeta^*\rangle-\langle\lambda\hat{\bxi},\hat{\bbeta}-\bbeta^*\rangle-\langle\nabla\cL(\bbeta^*)+h'(\bbeta),\hat{\bbeta}-\bbeta^*\rangle
\end{align*}

We write $\cD_{\tilde{\cL}}^s(\hat{\bbeta}^*,\bbeta^*)=\sum\limits_{j=1}^p\cD_{\tilde{\cL}}^s(\hat{\bbeta}^*,\bbeta^*)_j$ by breaking the inner products into the sum of entry-wise products.

\begin{enumerate}
\item When $j\in\{j: \abs{\beta^*_j}>a_1\lambda_j\}:=\cE_2$, $h'(\beta_j)=-\lambda_j\xi_j$ where $\xi_j\in\partial\abs{\beta_j}$.
\begin{equation}
\begin{split}
\cD_{\tilde{\cL}}^s(\hat{\bbeta},\bbeta^*)_j&=\langle\nabla\tilde{\cL}(\hat{\bbeta})_j+\lambda_j\hxi_j,\hat{\beta}_j-\beta^*_j\rangle-\langle\lambda_j\hxi_j,\hat{\beta}_j-\beta^*_j\rangle-\langle\nabla\cL(\bbeta^*)_j+h'(\beta^*_j),\hbeta_j-\beta^*_j\rangle\\
&=\langle\nabla\tilde{\cL}(\hat{\bbeta})_j+\lambda_j\hxi_j,\hat{\beta}_j-\beta^*_j\rangle+\langle -h'(\beta_j^*)-\lambda_j\hxi_j,\hat{\beta}_j-\beta^*_j\rangle-\langle\nabla\cL(\bbeta^*)_j-,\hbeta_j-\beta^*_j\rangle\\
&=\langle\nabla\tilde{\cL}(\hbbeta)_j+\lambda_j\hxi_j,\hbeta_j-\beta^*_j\rangle+\langle\lambda_j\xi_j^*-\lambda_j\hxi_j,\hat{\beta}_j-\beta^*_j\rangle-\langle\nabla\cL(\bbeta^*)_j,\hat{\beta}_j-\beta^*_j\rangle \\
&\leq \langle\nabla\tilde{\cL}(\hbbeta)_j+\lambda_j\hxi_j,\hbeta_j-\beta^*_j\rangle-\langle\nabla\cL(\bbeta^*)_j,\hat{\beta}_j-\beta^*_j\rangle
\end{split}
\label{eq:c2}
\end{equation}
The last inequality is due to $\langle\lambda_j\xi_j^*-\lambda_j\hxi_j,\hat{\beta}_j-\beta^*_j\rangle\leq 0$.

\item When $j\in\cE_1\cap S=\left\{j: 0<\abs{\beta^*_j}\leq a_1\lambda_j\right\}$, $p'_{\lambda}(\beta_j)\in[[0,\lambda\xi_j]]$ where $\xi_j\in\partial\abs{\beta_j}$. Here we use $[[x,y]]$ to denote $[\min(x,y),\max(x,y)]$ for $x,y\in\RR$. It follows that $h'(\beta)_j\in[[0,-\lambda_j\xi_j^*]]$ thus
\begin{align*}
\langle-\lambda_j\hxi_j-h'(\beta_j),\hbeta_j-\beta_j^*\rangle\leq &\max\left(\langle\lambda_j(\xi_j^*-\hxi_j),\hbeta_j-\beta_j^*\rangle,\langle-\lambda_j\hxi_j,\hbeta_j-\beta_j^*\rangle\right)\\
\leq &\max\left(0,\abs{\langle-\lambda_j,\hbeta_j-\beta_j^*\rangle}\right)=\abs{\langle\lambda_j,\hbeta_j-\beta_j^*\rangle}.
\end{align*}
Therefore,
\begin{equation}
\begin{split}
\cD_{\tilde{\cL}}^s(\hat{\bbeta},\bbeta^*)_j&\leq \langle\nabla\tilde{\cL}(\hat{\bbeta})_j+\lambda_j\hxi_j,\hat{\beta}_j-\beta^*_j\rangle +\abs{\langle\lambda_j,\hbeta_j-\beta_j^*\rangle}-\langle\nabla\cL(\bbeta^*)_j,\hat{\beta}_j-\beta^*_j\rangle\\
\end{split}
\label{eq:c3}
\end{equation}

\item When $j\in \cE_1\cap S^c=S^c=\{j: \abs{\beta^*_j}=0\}$, $h'(\bbeta)_j=0$.
\begin{equation}
\begin{split}
\cD_{\tilde{\cL}}^s(\hat{\bbeta},\bbeta^*)_j&=\langle\nabla\tilde{\cL}(\hat{\bbeta})_j+\lambda_j\hxi_j,\hat{\beta}_j-\beta^*_j\rangle-\langle\lambda_j\hxi_j,\hat{\beta}_j-\beta^*_j\rangle-\langle\nabla\cL(\bbeta^*)_j,\hat{\beta}_j-\beta^*_j\rangle\\
\end{split}
\label{eq:c4}
\end{equation}

\end{enumerate}

Sum up across all indices using (\ref{eq:c2}), (\ref{eq:c3}) and (\ref{eq:c4}),

\begin{align*}
\cD_{\tilde{\cL}}^s(\hat{\bbeta},\bbeta^*)=&\sum\limits_{j=1}^p\cD_{\tilde{\cL}}^s(\hat{\bbeta},\bbeta^*)_j\cdot(\mathds{1}_{j\in\cE_2}+\mathds{1}_{j\in\cE_1\cap S}+\mathds{1}_{j\in\cE_1\cap S^c}) \\
\leq & \langle\nabla\tilde{\cL}(\hat{\bbeta})+\lambda\hbxi,\hbbeta-\bbeta^*\rangle -\langle\nabla\cL(\bbeta^*),\hbbeta-\bbeta^*\rangle\\
&+\lambda\norm{(\hbbeta-\bbeta^*)_{\cE_1\cap S}}_1-\langle\lambda\hbxi_{S^c},(\hbbeta-\bbeta^*)_{S^c}\rangle
\end{align*}

We bound the four terms in the above inequality separately.

\begin{equation}
\begin{split}
&\langle\nabla\tilde{\cL}(\hat{\bbeta})+\lambda\hbxi,\hbbeta-\bbeta^*\rangle\\
=&\langle\nabla\tilde{\cL}(\hat{\bbeta})_{S}+\lambda\hbxi_{S},\hat{\bbeta}_{S}-\bbeta^*_{S}\rangle+\langle\nabla\tilde{\cL}(\hat{\bbeta})_{S^c}+(\lambda\hbxi)_{S^c},\hat{\bbeta}_{S^c}-\bbeta^*_{S^c}\rangle \\
\leq&\norm{\bu_{S}}_2\norm{(\hat{\bbeta}^*-\bbeta^*)_{S}}_2+\norm{\bu_{S^c}}_{\infty}\norm{(\hat{\bbeta}^*-\bbeta^*)_{S^c}}_1 \\
\leq& \varepsilon\sqrt{s}\cdot\norm{(\hat{\bbeta}^*-\bbeta^*)_{S}}_2+\varepsilon\cdot\norm{(\hat{\bbeta}^*-\bbeta^*)_{S^c}}_1
\end{split}
\label{eq:c7}
\end{equation}
The above inequality holds by taking $\inf$ over all $\hbxi\in\partial\norm{\hbbeta}_1$.

\begin{equation}
\langle\nabla\cL(\bbeta^*),\hbbeta-\bbeta^*\rangle\geq-\norm{\nabla\cL(\bbeta^*)_{S}}_2\cdot\norm{(\hat{\bbeta}^*-\bbeta^*)_{S}}_2-\norm{\nabla\cL(\bbeta^*)_{S^c}}_{\infty}\cdot\norm{(\hat{\bbeta}^*-\bbeta^*)_{S^c}}_1
\label{eq:c8}
\end{equation}

\begin{equation}
\label{eq:c9}
\lambda\norm{(\hbbeta-\bbeta^*)_{\cE_1\cap S}}_1\leq \lambda\sqrt{\abs{\cE_1\cap S}}\cdot\norm{(\hbbeta-\bbeta^*)_{\cE_1\cap S}}_2\leq \lambda\sqrt{\abs{\cE_1\cap S}}\cdot\norm{(\hbbeta-\bbeta^*)_{S}}_2
\end{equation}

Notice that $\langle\lambda\hbxi_{S^c},(\hat{\bbeta}-\bbeta^*)_{S^c}\rangle\geq 0$, we have
\begin{equation}
\langle\lambda\hbxi_{S^c},(\hbbeta-\bbeta^*)_{S^c}\rangle=\lambda\cdot\norm{(\hat{\bbeta}-\bbeta^*)_{S^c}}_1.
\label{eq:c10}
\end{equation}

Therefore, by (\ref{eq:c7}), (\ref{eq:c8}), (\ref{eq:c9}) and (\ref{eq:c10}),
\begin{equation}
\begin{split}
\cD_{\tilde{\cL}}^s(\hat{\bbeta},\bbeta^*)\leq& \left(\varepsilon\sqrt{s}+\norm{\nabla\cL(\bbeta^*)_{S}}_2+\lambda\sqrt{\abs{\cE_1\cap S}}\right)\cdot\norm{(\hat{\bbeta}-\bbeta^*)_{S}}_2\\
&-\left(\lambda-\norm{\nabla\cL(\bbeta^*)_{S^c}}_{\infty}-\varepsilon\right)\cdot\norm{(\hat{\bbeta}-\bbeta^*)_{S^c}}_1 \\
\leq&  \left(\varepsilon\sqrt{s}+\norm{\nabla\cL(\bbeta^*)_{S}}_2+\lambda\sqrt{\abs{\cE_1\cap S}}\right)\cdot\norm{(\hat{\bbeta}-\bbeta^*)_S}_2\\
\end{split}
\label{eq:c11}
\end{equation}

Combine (\ref{eq:cc}) and (\ref{eq:c11}), we have
\begin{equation}
\norm{\hat{\bbeta}^*-\bbeta^*}_2\leq \kappa_-^{-1}\left(\varepsilon\sqrt{s}+\norm{\nabla\cL(\bbeta^*)_{S}}_2+\lambda\sqrt{\abs{\cE_1\cap S}}\right)
\label{eq:c12}
\end{equation}

Since $\norm{\nabla\cL(\bbeta^*)}_{\infty}+\varepsilon\leq \lambda/2$, we have
\[
\norm{\hat{\bbeta}^*-\bbeta^*}_2\leq\kappa_-^{-1}\left(\lambda/2\cdot\sqrt{s}+\lambda\cdot\sqrt{s}\right)=\frac{3}{2}\lambda\sqrt{s}
\]
Using similar techniques as the proof of Proposition \ref{prop:contraction}, we show that $\norm{\hat{\bbeta}^*-\bbeta^*}_1\lesssim\lambda s<r$.
The definition of $\hbbeta^*$ suggests that $\hbbeta^*=\hbbeta$, since otherwise we have $\norm{\hbbeta^*-\bbeta^*}_1=r$. Therefore,  \eqref{eq:c12} becomes \[\norm{\hat{\bbeta}-\bbeta^*}_2\leq \kappa_-^{-1}\left(\varepsilon\sqrt{s}+\norm{\nabla\cL(\bbeta^*)_{S}}_2+\lambda\sqrt{\abs{\cE_1\cap S}}\right)\]

\end{proof}
\end{section}

\begin{section}{Proof of Theorem \ref{thm:strongoracle}}
\begin{proof}[Proof of Theorem \ref{thm:strongoracle}]
We prove the theorem in three steps: first we show that $\left\{\bbeta^k\right\}_{k=1}^{\infty}$ converges to the global optimum $\hbbeta$ in (\ref{eqloss}); then we prove that $\hbbeta=\widehat\bbeta^0$; finally, we illustrate that for $\bbeta^k$ sufficiently close to $\widehat\bbeta^0$, $\bbeta^{k+1}=\widehat \bbeta^0$.

According to Lemma \ref{lem:diff2} and Lemma \ref{lem:fg}, we know that
\begin{align*}
\norm{\bbeta^k-\bbeta^{k-1}}_2\leq\frac{2}{\phi}\left(F(\bbeta^k)-F(\hbbeta)\right)\leq\frac{2}{\phi}\left(1-\frac{1}{4\gamma_u\kappa}\right)^k\left(F(\ti{\bbeta}^{(2,0)})-F(\hbbeta)\right).
\end{align*}
Therefore $\left\{\bbeta^k\right\}_{k=1}^{\infty}$ converge as $k$ grows and $\left\{F(\bbeta^k)\right\}_{k=1}^{\infty}$ converges to $F(\hbbeta)$.

Then we show $\hbbeta=\widehat\bbeta^0$ with contradiction. Since $\norm{\nabla\cL(\widehat\bbeta^0)}_{\infty}<\lambda/2$ and $\lim\limits_{x\rightarrow 0}\abs{p'_{\lambda}(x)}=\lambda$, we have $\nabla F(\widehat\bbeta^0)=\nabla\cL(\widehat\bbeta^0)+p_{\lambda}(\widehat\bbeta^0)={\bf 0}$, 
which suggests that $\widehat\bbeta^0$ is a local minimizer of (\ref{eqloss}).
Using similar arguments in the proof of Proposition \ref{prop:decomposition2} and Proposition \ref{prop:stage1sparsity}, we derive that
\[\norm{\hbbeta}_0\leq\ti{s}\quad\text{and}\quad\norm{\hbbeta-\bbeta^*}_1\lesssim \lambda s\leq r,\] i.e. the global minimizer of (\ref{eqloss}) $\hbbeta$ is also within the sparsity cone $\cC(s+\ti{s},r)$ around $\bbeta^*$. Suppose that $\widehat\bbeta^0\neq\hbbeta$ and $F(\hbbeta)<F(\widehat\bbeta^0)$. The convexity of $F(\cdot)$
suggests that for any $\alpha\in(0,1)$, $F(\alpha\hbbeta+(1-\alpha)\widehat\bbeta^0)\leq\alpha F(\hbbeta)+(1-\alpha)F(\widehat\bbeta^0)<F(\widehat\bbeta^0)$. Take $\alpha$ sufficiently small, this contradicts with $F(\widehat\bbeta^0)$ being a local minimum. Hence, the global minimizer is unique and it is the same as the oracle estimator.

Finally, we need to show that the convergence in distance leads to strict equivalence. We want to show that there exists a sufficiently small $\delta>0$ such that if $\norm{\bbeta^{k-1}-\widehat\bbeta^0}_2\leq\delta$ then $\beta^k_j=0$ for $j\notin S$. Let $\delta_1=\inf\left\{x: p'_{\lambda}(x)/\lambda<2/3\right\}.$ The continuity of $p'_{\lambda}(\cdot)$ suggests that $\delta_1>0$. Take $\delta=\min\left(\delta_1,\gamma_u^{-1}\rho_+^{-1}\lambda/6\right)$.

Suppose $\norm{\bbeta^{k-1}-\widehat\bbeta^0}_2\leq\delta$ but there exists $j\in S^c$ such that $\bbeta_j^k\neq 0$, we must have
\begin{equation}
\abs{\beta^{k-1}_j-\frac{\nabla\ti{\cL}(\bbeta^{k-1})_j}{\phi}}>\frac{\lambda}{\phi}.
\label{eq:nonzero}
\end{equation}
Meanwhile,
\begin{equation}
\begin{split}
\norm{\nabla\ti{\cL}(\bbeta^{k-1})}_\infty=& \norm{\nabla\cL(\bbeta^{k-1})+p'_{\lambda}(\bbeta^{k-1})-\lambda\bxi^{k-1}}_\infty \\
\leq & \norm{\nabla\cL(\bbeta^{k-1})}_\infty+\norm{p'_{\lambda}(\bbeta^{k-1})-\lambda\bxi^{k-1}}_\infty \\
\leq & \frac{5\lambda}{6}
\end{split}
\label{eq:gradinfty}
\end{equation}
and
\begin{equation}
\abs{\beta^{k-1}_j}<\gamma_u^{-1}\rho_+^{-1}\lambda/6<\lambda/(6\phi).
\label{eq:closeness}
\end{equation}

\eqref{eq:gradinfty} and \eqref{eq:closeness} together contradict \eqref{eq:nonzero}, which means that for all $j\in S^c$ we have $\beta^k_j=0$. This also holds for all the steps after the $k$-th.

Therefore, an exact solution in the second stage would achieve the global minimum in the end.
 \end{proof}

\end{section}

\begin{section}{Proof of Theorem \ref{thm:tscomplexity}}

\begin{proof}[Proof of Theorem \ref{thm:tscomplexity}]
The proof for the first-stage complexity could be found in the proof of Proposition 4.5 in \cite{fan2018lamm}. Hence we only prove the second-stage complexity here. Apply Lemma \ref{lem:stopping2},
	we obtain 	\$
	\omega_{\lambda}(\btt^{k+1}) &\leq (1+\gamma_u)\rho_+\|\btt^{k+1}-\btt^{k}\|_2,
	\$
which, combining with Lemma \ref{lem:diff2}, yields
	\$
	\omega_{\blam}(\btt^{k+1})&\leq (1+\gamma_u)\rho_+\sqrt{\frac{2}{\phi}(F(\btt^{k})-F(\btt^{k+1}))}\\
	&\leq (1+\gamma_u)\rho_+\sqrt{2\rho_-^{-1}(F(\btt^{k})-F(\btt^{k+1}))},
	\$
where we use $\rho_-\leq \phi\leq \gamma_u\phi$ in the last inequality.
Since the sequence $\{F(\btt^{k})\}_{k=0}^\infty$ decrease monotonically, we obtain
	\$
	\omega_{\blam}(\btt^{k+1})&\!\leq\! (1+\gamma_u)\rho_+\sqrt{2\rho_-^{-1}(F(\btt^{k})\!-\!F(\tbt^2))}\\
&\!\leq\! (1+\gamma_u)\rho_+\sqrt{2\rho_-^{-1}\Big(1\!-\!\frac{1}{4\gamma_u\kappa}\Big)^k\left(F(\bbeta^{k})\!-\!F(\tbt^2)\right)}\\
&\!\leq\!(1+\gamma_u)\rho_+\sqrt{3\kappa_-^{-1}\rho_-^{-1}\Big(1-\frac{1}{4\gamma_u\kappa}\Big)^k\lambda^2s} \\
& \leq\sqrt{3}(1+\gamma_u)\kappa\sqrt{\Big(1-\frac{1}{4\gamma_u\kappa}\Big)^k\lambda^2s}
	\$
where $\kappa=\rho_+/\kappa_-$, the second inequality is due to Lemma \ref{lem:fg}, and the last one due to Lemma \ref{lem:3.15}.
Therefore, to ensure that $\btt^{k+1}$ satisfies $\omega_{\lambda}(\btt^{k+1})\leq \varepsilon$, it suffices to choose $k$ such that 
\$
\!\sqrt{3}(1+\gamma_u)\kappa\sqrt{\Big(1-\frac{1}{4\gamma_u\kappa}\Big)^k\lambda^2s}\leq \varepsilon.
\$
Equivalently, we obtain
\$
k\geq 2C'\log\Big(C''\frac{\lambda\sqrt{s}}{\varepsilon}\Big),
\$
where
$
C'=2/\log(4\gamma_u\kappa/\{4\gamma_u\kappa-1)\},~
C''=\sqrt{3}(1+\gamma_u)\kappa.
$
\end{proof}

\end{section}

\end{document}